\documentclass[11pt]{article}

\usepackage{amsmath, amssymb, amsthm, mathtools}       
\usepackage{bm, enumerate, mathrsfs, tensor, upgreek}  
\usepackage{indentfirst, emptypage, lmodern, setspace} 
\usepackage{caption, figsize, multirow, subfigure}     
\usepackage[greek,english]{babel}                      
\usepackage{csquotes}
\usepackage{authblk}
\usepackage{hyperref}
\usepackage[dvipsnames]{xcolor}
\usepackage{tikz}
\usepackage{adjustbox}
\usepackage{graphicx}
\usepackage[normalem]{ulem}
\usepackage{booktabs}   
\usepackage{float}      
\usetikzlibrary{fadings, decorations.pathreplacing}

\usepackage[linesnumbered,ruled,vlined]{algorithm2e}

\usepackage[margin=0.7in]{geometry}
\theoremstyle{plain}
\newtheorem{thm}{Theorem}[section]
\theoremstyle{definition}

\newtheorem{rem}[thm]{Remark}
\newtheorem{lem}[thm]{Lemma}
\newtheorem{prop}[thm]{Proposition}

\newcommand{\stkout}[1]{\ifmmode\text{\sout{\ensuremath{#1}}}\else\sout{#1}\fi}
\newcommand{\opnorm}[1]{\left\lVert #1 \right\rVert_{\mathrm{op}}}

\newcommand{\RR}{\mathbb{R}}

\newcommand{\T}{\mathrm{T}}

\newcommand{\K}{\mathbb{K}}

\numberwithin{equation}{section}
\graphicspath{{./FIGURES/HD/}}

\usepackage[
backend=biber,
giveninits=true,
style=numeric,
citestyle=numeric,
doi=false,
isbn=false,
url=false,
eprint=false,
maxbibnames=99
]{biblatex}
\title{
Optimal finite element error estimates and Newton convergence for a
quasilinear elliptic problem with mixed boundary conditions
}
\author[1,2,\thanks{E-mail: \texttt{mihai.bucataru@fmi.unibuc.ro;} \href{https://orcid.org/0000-0002-6503-2084}{ORCID ID: 0000-0002-6503-2084}}]
{Mihai Bucataru}
\affil[1]{\small Department of Mathematics, Faculty of Mathematics and Computer Science, University of Bucharest, 14~Academiei, 010014~Bucharest, Romania}
\affil[2]{\small ``Gheorghe Mihoc~--~Caius Iacob" Institute of Mathematical Statistics and Applied Mathematics of the Romanian Academy, 13~Calea 13 Septembrie, 050711~Bucharest, Romania}
\date{}

\begin{document}
\maketitle
\begin{abstract}\noindent The article studies finite element approximations of a quasilinear elliptic heat-conduction problem with inhomogeneous mixed boundary conditions. The conductivity tensor is matrix-valued, anisotropic, possibly nonsymmetric, and dependent on both position and temperature, rendering the problem nonlinear, nonmonotone, and nonpotential. The nonlinear algebraic system arising from the Galerkin discretization is solved using Newton’s method, with a posteriori guarantees provided by a computable Newton–Kantorovich criterion and a mesh-dependent stopping rule that ensures that the algebraic error is asymptotically negligible relative to the discretization error. Since the discrete solution need not be unique, we prove that every discrete solution satisfies the optimal convergence rates \(O(h)\) in the \(H^1\)-norm and \(O(h^2)\) in the \(L^2\)-norm. The analysis combines mixed-boundary elliptic regularity with an Aubin–Nitsche duality argument adapted to the quasilinear setting. Numerical experiments in two and three dimensions confirm the predicted convergence rates and demonstrate the feasibility of the proposed criterion.
\end{abstract}
\section{Introduction}

Quasilinear elliptic problems with solution-dependent anisotropic diffusion arise naturally in models of steady-state heat conduction in nonlinear inhomogeneous media, where the material response may vary with both position and temperature. The finite element approximation of nonlinear elliptic boundary value problems has been extensively studied, especially when the associated nonlinear operator is strongly monotone and Lipschitz continuous. In such a framework, an analogue of Céa’s lemma is available and optimal $H^1$-error estimates for Lagrange finite elements can be derived by rather standard arguments; see, for instance, \cite{Ciarlet, FeistauerSobotikova90,FeistauerZenisek88,Krizek,Zenisek90}. The problem considered in the present paper belongs, however, to a more delicate class. The diffusion tensor is matrix-valued, anisotropic, possibly nonsymmetric, and temperature-dependent, so the corresponding operator is in general neither monotone nor potential. Moreover, for $d>1$, the usual Kirchhoff transformation (see \cite{Krizek}) cannot reduce the equation to a linear problem, even when the conductivity is independent of the spatial variable, because the conductivity is a matrix rather than a scalar function. This structural feature makes both the theoretical and numerical analysis substantially more involved. One-dimensional examples illustrating the nonmonotone and nonpotential character of this class of problems were discussed in \cite{HKM94}.

Several analytical results are known for quasilinear elliptic problems of this type. Existence results for weak solutions, under various boundary conditions, were obtained using compactness and weak-continuity methods in \cite{Francu94,Necas}, while uniqueness and comparison results for related nonpotential and nonmonotone problems were investigated in \cite{HK93, Hlavacek97}. In particular, the Galerkin approximation of a quasilinear nonpotential elliptic problem of nonmonotone type was studied in \cite{HKM94}, where existence of discrete solutions was obtained by Brouwer’s fixed-point theorem and convergence of Galerkin approximations was proved, although without deriving convergence rates. The uniqueness of the discrete solution is a delicate issue in this setting. Indeed, early works already provided only restrictive sufficient conditions for uniqueness, and examples of André and Chipot showed that uniqueness may fail at the discrete level even when the corresponding continuous problem is uniquely solvable \cite{AndreChipot96Remark,AndreChipot96}. More recently, Pollock and Zhu proved uniqueness of continuous piecewise linear finite element solutions for scalar nonmonotone quasilinear diffusion problems in one and two dimensions, under local computable bounds on the variation of the discrete solution over each element, thereby avoiding the older requirement of a globally fine mesh \cite{PollockZhu18}. Their result clarifies the uniqueness mechanism in the scalar case and is particularly relevant for adaptive refinement, since the condition can be checked elementwise. However, it does not directly cover the general anisotropic matrix-valued conductivity considered here, nor the three-dimensional setting.

From the numerical point of view, Douglas and Dupont \cite{DouglasDupont75} proved optimal finite element convergence for scalar nonlinear diffusion tensors of the form \(A(x,u)=\lambda(x,u)I\). Related developments include \(L^\infty\)-error estimates \cite{Nitsche1977}, mixed finite element methods for quasilinear elliptic problems \cite{Milner1985}, and convergence results for nonlinear mixed finite element methods \cite{Chen1989}. This scalar theory was later extended in \cite{Liu96,Liu1997} to smooth uniformly positive definite matrix-valued conductivities \(A(x,u)\), still in the homogeneous Dirichlet setting, yielding optimal convergence rates. Superconvergence and postprocessing-based a posteriori estimators were subsequently developed in \cite{Wahlbin1995,LiuLiuKrizekLinZhang2004}, showing that refined interpolation and recovery techniques can provide higher-order approximations and computable error indicators.

Later work broadened the numerical methodology in several directions. Finite volume element approximations and residual-type a posteriori estimators were studied in \cite{BiGinting2009}, while discontinuous Galerkin methods and their a posteriori analysis were developed in \cite{GudiPani2007,BiGinting2013}. Adaptive conforming finite element methods for nonmonotone quasilinear elliptic problems were investigated in \cite{GuoBi2021,GuoBi2023}, and hybrid high-order methods on general polytopal meshes were proposed in \cite{GudiMallikPramanick2022}. 
A complementary solver-oriented line of work studies how accurately the nonlinear and linear algebraic systems must be solved relative to the discretization error. Ern and Vohralík developed an adaptive inexact Newton framework with a posteriori stopping criteria that distinguish discretization, linearization, and algebraic errors \cite{ErnVohralik2013}. Related ideas were extended to unsteady nonlinear advection--diffusion problems in \cite{DolejsiErnVohralik2013} and, more recently, to semismooth Newton methods for nonsmooth constrained problems in \cite{DabaghiMartinVohralik2020}. These works provide a systematic framework for balancing discretization, linearization, and algebraic errors within adaptive nonlinear solution algorithms.

The present paper studies the conforming finite element approximation of a quasilinear elliptic model for steady-state anisotropic heat conduction with temperature-dependent conductivity. The problem is posed with inhomogeneous mixed Dirichlet--Neumann boundary conditions, and the conductivity tensor \(\K(x,u)\) is allowed to depend on both the spatial variable and the unknown temperature. In contrast with scalar nonlinear diffusion, the tensor is matrix-valued, anisotropic and possibly nonsymmetric, so that the associated operator is generally neither monotone nor potential. Consequently, the analysis cannot rely on a Kirchhoff transformation, on a variational minimization principle, or on a standard monotonicity-based Céa argument. We consider conforming linear finite element discretizations and study both the finite element error and the nonlinear algebraic system obtained after discretization. Since uniqueness of the nonlinear Galerkin solution is not available in the present generality, the a priori error analysis is developed for arbitrary discrete Galerkin solutions. The proof combines mixed-boundary \(H^2\)-regularity, a preliminary nonlinear \(H^1\)-estimate, and an Aubin--Nitsche duality argument adapted to the nonsymmetric quasilinear setting. In addition, Newton's method is analyzed at the discrete level through a computable Newton--Kantorovich criterion, and the algebraic error produced by stopping the nonlinear iteration is related to the mesh size. The theoretical results are complemented by numerical experiments in two and three spatial dimensions, which confirm the predicted convergence rates, illustrate the behavior of the Newton corrections, and demonstrate the practical verification of the Newton--Kantorovich condition.

The main contributions of the present study are summarized as follows:
\begin{enumerate}
\item We treat an inhomogeneous mixed Dirichlet--Neumann boundary value problem for a quasilinear elliptic equation with matrix-valued, anisotropic, possibly nonsymmetric and temperature-dependent conductivity. This extends the optimal finite element analysis beyond the homogeneous Dirichlet framework most commonly considered for this class of nonmonotone problems.

\item We derive the finite element estimates without assuming uniqueness of the nonlinear discrete solution. More precisely, the optimal error bounds are proved for arbitrary conforming \(P_1\) Galerkin solutions, which is essential because global uniqueness of the discrete nonlinear problem is not guaranteed under the present assumptions.

\item We establish and use the \(H^2\)-regularity needed for linear elliptic mixed boundary value problems with \(W^{1,\infty}\) coefficients in the present separated Dirichlet–Neumann configuration. This regularity result is applied not only to the quasilinear state equation, but also to the adjoint problem required in the \(L^2\)-error analysis.

\item We prove the optimal convergence rates
\[
\|u-u_h\|_{H^1(\Omega)} = O(h),
\qquad
\|u-u_h\|_{L^2(\Omega)} = O(h^2),
\]
for conforming piecewise linear approximations of the mixed anisotropic quasilinear problem.

\item We complement the a priori finite element analysis with a discrete Newton--Kantorovich theory for the nonlinear algebraic system. The resulting convergence criterion is expressed in terms of computable quantities: the inverse Jacobian, an explicit Lipschitz bound for the discrete derivative, and the Newton radius.

\item We show that the algebraic error generated by a finite number of Newton iterations can be controlled by a mesh-dependent stopping criterion. Under this stopping rule, the algebraic error does not deteriorate the optimal \(H^1\)- and \(L^2\)-finite element convergence rates.

\end{enumerate}

The article is organized as follows. Section~2 introduces the quasilinear heat-conduction model, the mixed Dirichlet--Neumann boundary conditions, the assumptions on the data and on the conductivity tensor, and the associated weak formulation. Section~3 discusses the continuous problem, including the variational setting, the existence framework, and the regularity assumptions needed in the subsequent analysis. Section~4 presents the conforming Galerkin discretization by piecewise linear finite elements, recalls the existence of discrete solutions, and emphasizes the absence of a general uniqueness result for the nonlinear discrete problem. Section~5 contains the finite element error analysis: after proving the required \(H^2\)-regularity for mixed elliptic problems with \(W^{1,\infty}\) coefficients, it derives a preliminary \(H^1\)-estimate, formulates the adjoint problem, and applies the Aubin--Nitsche argument to obtain the optimal \(H^1\)- and \(L^2\)-error bounds. Section~6 is devoted to the nonlinear algebraic system and to Newton's method; it establishes a Newton--Kantorovich convergence criterion, derives an explicit Lipschitz estimate for the discrete Jacobian, and explains how the algebraic error can be controlled without affecting the finite element rates. Section~7 presents numerical experiments in two and three spatial dimensions, illustrating the theoretical convergence rates and the practical verification of the Newton--Kantorovich condition. The appendix collects auxiliary finite-dimensional and mesh-dependent estimates used in the Newton analysis.

\section{Problem setting}

Consider a solid that occupies a bounded domain $\Omega \subset \mathbb{R}^d$, $d\in\{2,3\}$,
with a $C^2$ boundary $\partial\Omega = \Gamma_D \cup \Gamma_N$, such that $\overline{\Gamma}_D \cap \overline{\Gamma}_N = \varnothing$ and $|\Gamma_D|>0$. We investigate the steady-state temperature distribution within this body, assuming anisotropic heat propagation and temperature-dependent conductivity, i.e., we seek $u \in H^1(\Omega)$ satisfying the following quasilinear BVP:
\begin{equation}\label{eq:strong_problem}
\left\{
\begin{aligned}
-\nabla\cdot\bigl(\K(x,u)\nabla u\bigr) &= f 
&& \text{in }\Omega,\\
u &= g 
&& \text{on }\Gamma_D,\\
\nu\cdot\bigl(\K(x,u)\nabla u\bigr) &= h 
&& \text{on }\Gamma_N,
\end{aligned}
\right.
\end{equation}
where $f\in L^2(\Omega)$ is the heat source, $g\in H^{1/2}(\Gamma_D)$ is the prescribed temperature, and
$h\in H^{-1/2}(\Gamma_N)$ is the prescribed normal heat flux. The thermal conductivity tensor $\K:\Omega \times \mathbb{R} \to \mathbb{R}^{d\times d}$ is a Carathéodory function, continuously differentiable with respect to its first argument, twice continuously differentiable with respect to its second argument, and satisfies the following conditions for a.e.\ $x \in \Omega$ and all $s \in \mathbb{R}$:
\begin{align}
&\opnorm{\K(x,s)} \le \Lambda, \label{K1} \\[4pt]
&\lambda |\xi|^2 \le \xi^\top \K(x,s)\,\xi,
\qquad \forall \xi \in \mathbb{R}^d, \label{K2} \\[4pt]
&\opnorm{\K_u(x,s)} + \opnorm{\K_{uu}(x,s)} \le C_u, \label{K3} \\[4pt]
&\|\nabla_x \K(x,s)\| \le C_x. \label{K4}
\end{align}
where $\Lambda,\lambda,C_u,C_x>0$, $\K_u(x,s)$ and $\K_{uu}(x,s)$ denote the first and
second derivatives of $\K(x,\cdot)$ with respect to $s$. By $\|\cdot\|_{\mathrm{op}}$ we denote the operator norm on $\mathbb{R}^{d\times d}$ induced by the Euclidean norm on $\mathbb{R}^d$, and for the tensor field $\nabla_x\K(x,s)$ we set
\[
\|\nabla_x\K(x,s)\|
:=
\left(
\sum_{i=1}^d
\|\partial_{x_i}\K(x,s)\|_{\mathrm{op}}^2
\right)^{1/2}.
\]

\section{Variational formulation}\label{seq:FEM}
Next, we consider the Hilbert space $V := \{ v\in H^1(\Omega) : v|_{\Gamma_D}=0 \}$, equipped with the seminorm $\|v\|_V := \|\nabla v\|_{L^2(\Omega)}$. We note that since $|\Gamma_D|>0$, the Poincar\'e inequality ensures that $\|\cdot\|_V$ is a norm on $V$ equivalent to the $H^1(\Omega)$-norm.

To define the weak formulation of \eqref{eq:strong_problem}, we introduce the parameter-dependent bilinear form $a(\,\cdot\,;\,\cdot\,,\,\cdot\,):H^1(\Omega)\times H^1(\Omega)\times V\to\mathbb{R}$ and the linear functional $\ell:V\to\mathbb{R}$ by
\begin{align}
a(w;u,v)
&:=\int_\Omega (\nabla u)^\top \K(x,w)\,\nabla v\,dx,
\qquad w,u\in H^1(\Omega),\ v\in V, \label{eq:def_a}\\
\ell(v)
&:=\int_\Omega f\,v\,dx+\langle h, v\rangle_{\Gamma_N},
\qquad v\in V, \label{eq:def_ell}
\end{align}
where $\langle\cdot,\cdot\rangle_{\Gamma_N}$ denotes the duality pairing between
$H^{-1/2}(\Gamma_N)$ and $H^{1/2}(\Gamma_N)$. We note that, owing to the uniform boundedness of $\K$ \eqref{K1}, the form $a(w;u,v)$ is well defined and continuous. Moreover, for a fixed $w \in H^1(\Omega)$, due to the ellipticity of $\K$ \eqref{K2}, $a(w;\cdot,\cdot)$ is uniformly $V$-elliptic and continuous.

Next, we let $u_g \in H^1(\Omega)$ be a lifting of $g$ such that $u_g |_{\Gamma_D} = g$. We are now able to present the variational formulation of \eqref{eq:strong_problem}, namely, to find $u \in H^1(\Omega)$ such that $u - u_g \in V$ and 
\begin{equation}\label{eq:weak_problem}
    a(u; u, v) = \ell(v), \quad \forall v \in V.
\end{equation}
The existence of a unique weak solution to \eqref{eq:weak_problem} is proven in \cite{HKM94}.

\section{Galerkin discrete formulation}
Let $\mathcal{T}_h$ be a conforming, shape-regular triangulation of $\Omega$, i.e., there exists a constant $\sigma \ge 1$ such that for every element $T \in \mathcal{T}_h$,
\begin{equation}\label{eq:shape-regularity}
    \frac{h_T}{\rho_T} \le \sigma,
\end{equation}
where $h_T = \mathrm{diam}(T)$ and $\rho_T$ is the radius of the largest inscribed ball in $T$. A global measure of such a triangulation $\mathcal{T}_h$ is the maximum diameter, $h \coloneqq \underset{T \in \mathcal{T}_h}{\mathrm{max}} \; h_T$. 

Let $V_h \subset V$ be the space of continuous, piecewise linear finite element functions that vanish on $\Gamma_D$. Further, let $u_{g,h} \in H^1(\Omega)$ be a discrete lifting of $g$ such that $u_{g,h}|_{\Gamma_D} = g_h$, where $g_h$ is an approximation of $g$ on $\Gamma_D$. The Galerkin approximation reads: find $u_h \in u_{g,h} + V_h$ such that
\begin{equation}\label{eq:weak_problem_discrete}
    a(u_h; u_h, v_h) = \ell(v_h), \quad \forall v_h \in V_h.
\end{equation}

We refer to \cite{HKM94} for an existence proof of the discrete solution via the Brouwer fixed-point theorem. In addition, it is shown in \cite[Theorem 2.9 (i)]{HKM94} that the sequence of discrete solutions $\{u_h\}$ is bounded in $H^1(\Omega)$ and converges weakly in $H^1$, although no error estimates are derived there.

Although sufficient (yet highly restrictive) conditions ensuring uniqueness are provided there, a general global uniqueness result remains open. Therefore, we do not impose additional assumptions to enforce uniqueness and derive the error estimates for any discrete solution $u_h$.

\section{Error estimates}

In this section, we extend the error analysis of \cite{Liu96}, originally developed for homogeneous Dirichlet data, to the present mixed inhomogeneous boundary setting. Throughout, we assume the following regularity of the boundary data:
\begin{equation}\label{eq:estimates-assumption}
    g \in H^{3/2}(\Gamma_D)\,, \qquad h \in H^{1/2}(\Gamma_N)\,.
\end{equation}
Let $u \in u_g + V$ be the unique weak solution of \eqref{eq:strong_problem}. In addition to \eqref{eq:estimates-assumption}, we assume that
\begin{equation}\label{eq:drift-assumption}
    \K_u(\cdot,u)\nabla u \in L^\infty(\Omega)^d .
\end{equation}
\begin{rem}
A sufficient condition for \eqref{eq:drift-assumption} is $u \in W^{1,\infty}(\Omega)$.
Indeed, by the uniform bound on $\K_u$ from \eqref{K3}, we have for a.e.\ $x \in \Omega$,
\[
|\K_u(x,u(x))\nabla u(x)|
    \le \opnorm{\K_u(x,u(x))}\,|\nabla u(x)|
    \le C_u |\nabla u(x)|,
\]
and therefore
\[
\|\K_u(\cdot,u)\nabla u\|_{L^\infty(\Omega)^d}
\le C_u \|\nabla u\|_{L^\infty(\Omega)^d}.
\]
\end{rem}
The following elliptic regularity proposition is key to deriving our error estimates.
\begin{prop}\label{prop:elliptic_regularity}
Let $A\in W^{1,\infty}(\Omega)^{d\times d}$ be uniformly elliptic, $\xi \in L^2(\Omega)$, $\psi \in H^{3/2}(\Gamma_D)$, and $\eta \in H^{1/2}(\Gamma_N)$. Finally, assume that the following mixed boundary value problem is well posed:
\begin{equation}
\left\{
\begin{aligned}
-\nabla\cdot(A\nabla\varphi) &= \xi &&\text{in }\Omega,\\
\varphi&=\psi &&\text{on }\Gamma_D,\\
\nu\cdot(A\nabla\varphi)&=\eta &&\text{on }\Gamma_N.
\end{aligned}
\right.
\end{equation}
Then the unique weak solution $\varphi \in H^1(\Omega)$ belongs to $H^2(\Omega)$ and satisfies
\[
\|\varphi\|_{H^2(\Omega)}
\le
C\bigl(
\|\xi\|_{L^2(\Omega)}
+\|\psi\|_{H^{3/2}(\Gamma_D)}
+\|\eta\|_{H^{1/2}(\Gamma_N)}
\bigr).
\]
In particular, under assumptions \eqref{eq:estimates-assumption} and \eqref{eq:drift-assumption}, the unique weak solution of \eqref{eq:strong_problem} satisfies $u \in H^2(\Omega)$ and
\[
\|u\|_{H^2(\Omega)}
\le
C\Bigl(
\|f\|_{L^2(\Omega)}
+\|g\|_{H^{3/2}(\Gamma_D)}
+\|h\|_{H^{1/2}(\Gamma_N)}
\Bigr).
\]
\end{prop}
\begin{proof}

Since $\overline{\Gamma}_D \cap \overline{\Gamma}_N = \varnothing$, one may cover
$\overline{\Omega}$ by finitely many open sets of three types: interior patches,
patches meeting the boundary only on $\Gamma_D$, and patches meeting the boundary only on
$\Gamma_N$. Let $\{\chi_j\}$ be a smooth partition of unity subordinate to such a covering.
For an interior patch, $\chi_j \varphi$ satisfies an elliptic equation with right-hand side in
$L^2$, and the standard interior regularity theorem \cite[Theorem 8.8]{GilbargTrudiger} yields $\chi_j \varphi\in H^2(\Omega)$.

For a patch meeting the boundary only on $\Gamma_D$,
$\chi_j \varphi$ satisfies an elliptic problem with Dirichlet boundary condition
\[
\chi_j \varphi = \chi_j \psi \quad \text{on } \Gamma_D \cap \operatorname{supp}\chi_j.
\]
Since $\psi \in H^{3/2}(\Gamma_D)$ and $\chi_j$ is smooth, we have
$\chi_j \psi \in H^{3/2}(\Gamma_D \cap \operatorname{supp}\chi_j)$. By the trace lifting theorem (see \cite[Theorem 2]{Marschall87}), we can reduce the problem to homogeneous Dirichlet data and apply \cite[Theorem 8.12]{GilbargTrudiger} to obtain $H^2$ regularity on these patches.

It remains to consider a patch meeting the boundary only on $\Gamma_N$.
After localization by $\chi_j$ and flattening the boundary by a local $C^2$
diffeomorphism, the problem is transformed to an elliptic equation in
a half-domain $U^+ = \{x \in U : x_d > 0\}$,
with conormal boundary condition
\[
Bv \coloneqq c_\nu\,\partial_{x_d} v + c_\tau\cdot\nabla_\tau v = g
\quad \text{on }\Gamma_U := \partial U^+ \cap \{x_d = 0\},
\]
where $c_\nu$ is bounded away from zero by ellipticity and
$g \in H^{1/2}(\Gamma_U)$.
Since $c_\nu$ is sufficiently regular and bounded away from zero,
we also have $g/c_\nu \in H^{1/2}(\Gamma_U)$.
Using a lifting guaranteed by \cite[Theorem 2]{Marschall87},
there exists $w \in H^2(U^+)$ such that
\[
w|_{\Gamma_U} = 0,
\qquad
\partial_{x_d} w|_{\Gamma_U} = g/c_\nu.
\]
Since $w|_{\Gamma_U} = 0$, all tangential trace derivatives vanish on $\Gamma_U$,
and therefore $Bw = g$ on $\Gamma_U$.
Thus $z := v - w$ satisfies an elliptic equation in $U^+$ with homogeneous
conormal boundary condition. Applying Grisvard's local $H^2$-regularity
argument for elliptic problems with homogeneous conormal boundary condition
(\cite[Section 2.2.2]{Grisvard}) to this weak solution,
we obtain $z \in H^2(U^+)$, hence also $v \in H^2(U^+)$.
Transforming back, we conclude that $\chi_j \varphi \in H^2(\Omega)$ on each Neumann
patch.

Summing over the finitely many patches, we infer
\[
\|\varphi\|_{H^2(\Omega)}
\le
C\bigl(
\|\xi\|_{L^2(\Omega)}
+\|\psi\|_{H^{3/2}(\Gamma_D)}
+\|\eta\|_{H^{1/2}(\Gamma_N)}
+\|\varphi\|_{H^1(\Omega)}
\bigr).
\]
Since this mixed problem is well posed, we obtain
\[
\|\varphi\|_{H^2(\Omega)}
\le
C\Bigl(
\|\xi\|_{L^2(\Omega)}
+\|\psi\|_{H^{3/2}(\Gamma_D)}
+\|\eta\|_{H^{1/2}(\Gamma_N)}
\Bigr).
\]

For the final statement, let $A(\cdot)=\K(\cdot,u(\cdot))$. Since $\K$ is differentiable with respect to the spatial variable and satisfies \eqref{K4}, the matrix field $A(x):=\K(x,u(x))$ satisfies
\[
\nabla A(x) = \nabla_x \K(x,u(x)) + \K_u(x,u(x)) \nabla u(x).
\]
Hence, by \eqref{K4} and \eqref{eq:drift-assumption}, we obtain $A \in W^{1,\infty}(\Omega)^{d\times d}$. Applying the first part of Proposition~\ref{prop:elliptic_regularity} to \eqref{eq:strong_problem}, with $\xi=f \in L^2(\Omega)$, $\psi=g \in H^{3/2}(\Gamma_D)$, and $\eta=h \in H^{1/2}(\Gamma_N)$,
we conclude that $u \in H^2(\Omega)$ and
\[
\|u\|_{H^2(\Omega)}
\le
C\Bigl(
\|f\|_{L^2(\Omega)}
+\|g\|_{H^{3/2}(\Gamma_D)}
+\|h\|_{H^{1/2}(\Gamma_N)}
\Bigr).
\qedhere
\]
\end{proof}

\begin{rem}\label{rem:strong-convergence}
We now verify that the assumptions of \cite[Theorem 2.9(ii)]{HKM94}
are satisfied in the present setting. Since $u\in H^2(\Omega)$
and $d\le 3$, the Sobolev embedding yields $u\in W^{1,q}(\Omega)$
for some $q>d$ (see \cite[Theorem 4.12]{Adams}). Moreover, for conforming finite element spaces on
regular families of triangulations, the approximation properties
(2.25)–(2.26) of \cite{HKM94} are satisfied by the standard nodal
interpolant (see Remark 2.8 therein). Therefore, Theorem 2.9(ii)
of \cite{HKM94} applies, and the sequence of discrete solutions
$\{u_h\}$ converges strongly in $H^1(\Omega)$
to the weak solution $u$.
\end{rem}

\begin{lem}\label{lem:interpolation-estimate}
Let $u_h \in u_{g,h} + V_h$ be a (not necessarily unique) discrete solution of \eqref{eq:weak_problem_discrete}. Assuming \eqref{eq:estimates-assumption} and \eqref{eq:drift-assumption}, 
there exists an interpolant $I_h u$ of $u$ and a constant $C = C(\Omega,\sigma) > 0$ such that, for any $h>0$,
\begin{equation}\label{eq:interpolation-estimate}
    \|u-I_h u\|_{H^1(\Omega)} \le Ch\bigl(\|u\|_{H^2(\Omega)} + \|g\|_{H^{3/2}(\Gamma_D)}\bigr)\,.
\end{equation}
\end{lem}

\begin{proof}
Since $\partial\Omega$ is $C^2$, the trace operator 
$\gamma_D:H^2(\Omega)\to H^{3/2}(\Gamma_D)$ admits a bounded right-inverse
$E_D:H^{3/2}(\Gamma_D)\to H^2(\Omega)$ (see \cite[Theorem 2]{Marschall87}). 
Hence, for $g\in H^{3/2}(\Gamma_D)$, we may choose a lifting $u_g := E_D g \in H^2(\Omega)$, satisfying
\[
\|u_g\|_{H^2(\Omega)} \le C \|g\|_{H^{3/2}(\Gamma_D)}.
\]
For a fixed $h > 0$, let $I_h (u - u_g) \in V_h$ be the nodal piecewise linear interpolant and set $u_{g,h} := I_h u_g$. We define
\[
I_h u := u_{g,h} + I_h(u-u_g).
\]
Since $u - u_g, u_g \in H^2(\Omega)$, the interpolation estimates of \cite[Theorem 4.4.20]{BrennerScott} yield
\begin{align*}
\|u-I_h u\|_{H^1(\Omega)} 
&\le \|(u-u_g) - I_h(u-u_g)\|_{H^1(\Omega)} + \|u_g - I_h u_g\|_{H^1(\Omega)} \le Ch\bigl(\|u-u_g\|_{H^2(\Omega)} + \|u_g\|_{H^2(\Omega)}\bigr) \\
&\le Ch\bigl(\|u\|_{H^2(\Omega)} + \|g\|_{H^{3/2}(\Gamma_D)}\bigr).
\end{align*}
\end{proof}

\begin{rem}\label{rem:lifting}
In practice, $u_{g,h}$ is not constructed as $I_h u_g$, since $u_g$ (and hence $I_h u_g$) is typically unknown. Instead, one defines $u_{g,h}$ as the piecewise linear interpolant whose nodal values coincide with $g$ on $\Gamma_D$ and vanish at the remaining nodes. 

This choice leads to an equivalent discrete formulation and, in practice, to the same computed solution $u_h$, see \cite[Section~5.3]{Gockenbach}. Although that discussion is given for the linear case, the same observation carries over to the present setting, since the nonlinearity depends only on the discrete solution $u_h$ and not on the particular choice of lifting.
\end{rem}

\begin{thm}\label{thm:energy-reduction}
Assuming \eqref{K1}--\eqref{K4}, \eqref{eq:estimates-assumption} and \eqref{eq:drift-assumption}  for $u$, the unique weak solution to \eqref{eq:strong_problem}, there exists $C = C\bigl(\|u\|_{H^2(\Omega)},\|g\|_{H^{3/2}(\Gamma_D)},\Omega,\sigma,\lambda,\Lambda,C_u\bigr) > 0$
 such that for every $h > 0$ and every discrete solution $u_h$ of \eqref{eq:weak_problem_discrete}, the error \(e_h:=u-u_h\) satisfies
\begin{equation}\label{eq:energy-reduction}
\|e_h\|_{H^1(\Omega)} \le C\bigl(h+\|e_h\|_{L^2(\Omega)}\bigr)\,,
\end{equation}
\end{thm}

\begin{proof}
Fix \(h>0\). By the triangle inequality and Lemma~\ref{lem:interpolation-estimate},
\begin{equation}\label{eq:tri-energy}
\|e_h\|_{H^1(\Omega)}
\le
\|u-I_hu\|_{H^1(\Omega)}+\|I_hu-u_h\|_{H^1(\Omega)}
\le
C_1 h + \|I_hu-u_h\|_{H^1(\Omega)}.
\end{equation}
where $C_1 = C_1(\|u\|_{H^2(\Omega)},\|g\|_{H^{3/2}(\Gamma_D)},\Omega, \sigma)$ and $I_h u$ is defined as in the proof of Lemma~\ref{lem:interpolation-estimate}.

Next, set \(\eta_h:=I_hu-u_h\in V_h\). Subtracting the continuous and discrete variational identities and testing with \(v_h=\eta_h\), we obtain
\begin{equation}\label{eq:eta-identity}
a(u_h;\eta_h,\eta_h)
=
a(u_h;I_hu,\eta_h)-a(u;u,\eta_h).
\end{equation}
Using the uniform ellipticity of $\K$ \eqref{K2}, Poincar\'e's inequality, and \eqref{eq:eta-identity}, we obtain
\begin{equation}\label{eq:eta-coerc}
\frac{\lambda}{C_P^2}\|\eta_h\|_{H^1(\Omega)}^2
\le
\lambda\|\nabla\eta_h\|_{L^2(\Omega)}^2
\le
a(u_h;\eta_h,\eta_h)
\le
|a(u_h;I_hu-u,\eta_h)|+|a(u_h;u,\eta_h)-a(u;u,\eta_h)|.
\end{equation}
The first term is bounded by continuity of \(a(u_h;\cdot,\cdot)\) and employing once again Lemma~\ref{lem:interpolation-estimate}:
\begin{equation}\label{eq:error_bound_1}
    |a(u_h;I_hu-u,\eta_h)|
\le
\Lambda \|u-I_hu\|_{H^1(\Omega)}\|\eta_h\|_{H^1(\Omega)}
\le
C_2 h \|\eta_h\|_{H^1(\Omega)},
\end{equation}
where \(C_2=C_2(\|u\|_{H^2(\Omega)},\|g\|_{H^{3/2}(\Gamma_D)},\Omega,\sigma,\Lambda)>0\).

For the second term, we write
\[
a(u_h;u,\eta_h)-a(u;u,\eta_h)
=
\int_\Omega (\nabla u)^\top\bigl(\K(x,u_h)-\K(x,u)\bigr)\nabla\eta_h\,dx.
\]
The boundedness of the derivative of $\K$ in \eqref{K3} implies, via the mean value theorem, the Lipschitz continuity of $\K$ with respect to its second argument. Combining this with H\"older's inequality, we obtain
\[
|a(u_h;u,\eta_h)-a(u;u,\eta_h)|
\le
C_u\|\nabla u\|_{L^6(\Omega)}\,\|e_h\|_{L^3(\Omega)}\,\|\nabla\eta_h\|_{L^2(\Omega)} .
\]
Since \(d\le 3\) and, by \eqref{eq:drift-assumption} together with the regularity assumptions on \(\K\), we have \(u\in H^2(\Omega)\), it follows that \(\nabla u\in H^1(\Omega)\). Therefore, the Sobolev embedding \(H^1(\Omega)\hookrightarrow L^6(\Omega)\) (\cite[Theorem 4.12]{Adams}) yields
\[
\|\nabla u\|_{L^6(\Omega)}\le C\|\nabla u\|_{H^1(\Omega)}\le C\|u\|_{H^2(\Omega)}.
\]
Moreover, by interpolation between \(L^2(\Omega)\) and \(L^6(\Omega)\), and using once again the embedding \(H^1(\Omega)\hookrightarrow L^6(\Omega)\),
\[
\|e_h\|_{L^3(\Omega)}\le C\|e_h\|_{L^2(\Omega)}^{1/2}\|e_h\|_{L^6(\Omega)}^{1/2} \le C\|e_h\|_{L^2(\Omega)}^{1/2}\|e_h\|_{H^1(\Omega)}^{1/2}.
\]
Hence,
\begin{equation}\label{eq:error_bound_2}
|a(u_h;u,\eta_h)-a(u;u,\eta_h)|
\le
C_3\|e_h\|_{L^2(\Omega)}^{1/2}\|e_h\|_{H^1(\Omega)}^{1/2}\,\|\eta_h\|_{H^1(\Omega)},
\end{equation}
where \(C_3=C_3(\|u\|_{H^2(\Omega)},\Omega,C_u)>0\).

Inserting the bounds \eqref{eq:error_bound_1} and \eqref{eq:error_bound_2} into \eqref{eq:eta-coerc}, and dividing by $\dfrac{\lambda}{C_P^2}\|\eta_h\|_{H^1(\Omega)}$ (as $\eta_h\neq 0$; if it is zero, the error estimate \eqref{eq:energy-reduction} coincides with \eqref{eq:interpolation-estimate}), we obtain
\[
\|\eta_h\|_{H^1(\Omega)} \le \frac{C_P^2}{\lambda} \left( C_2 h + C_3 \|e_h\|_{L^2(\Omega)}^{1/2}\|e_h\|_{H^1(\Omega)}^{1/2} \right).
\]
Substituting the above estimate into \eqref{eq:tri-energy} and absorbing all constants into a generic constant \(C>0\), we obtain
\[
\|e_h\|_{H^1(\Omega)}
\le
C\Bigl(h+\|e_h\|_{L^2(\Omega)}^{1/2}\|e_h\|_{H^1(\Omega)}^{1/2}\Bigr).
\]
If \(\|e_h\|_{H^1(\Omega)}=0\), the estimate is trivial. Otherwise, by the AM-GM inequality,
\[
\|e_h\|_{H^1(\Omega)}
\le
Ch+\frac12 \|e_h\|_{H^1(\Omega)}+\frac{C^2}{2}\|e_h\|_{L^2(\Omega)},
\]
and therefore
\[
\|e_h\|_{H^1(\Omega)} \le 2Ch + C^2\|e_h\|_{L^2(\Omega)}.
\]
Redefining \(C\) yields \eqref{eq:energy-reduction}.
\end{proof}

To obtain the $L^2$-estimate, we next derive the adjoint problem associated with the linearization of \eqref{eq:strong_problem} and establish its well-posedness, after which we apply the Aubin--Nitsche argument. To do so, for a direction $v \in V \cap H^2(\Omega)$, we compute the G\^ateaux derivatives of the nonlinear differential operators $\mathcal L(u)\coloneqq -\nabla\cdot\bigl(\K(x,u)\nabla u\bigr)$ and $\mathcal N(u) \coloneqq \nu\cdot\bigl(\K(x,u)\nabla u\bigr)$:
\begin{align*}
D\mathcal L(u)\,v
&= -\nabla\cdot\Bigl(\K(x,u)\nabla v + v\,\K_u(x,u)\nabla u\Bigr)\qquad \text{on }\Omega\,,\\
D\mathcal N(u)\,v
&= \nu\cdot\Bigl(\K(x,u)\nabla v + v\,\K_u(x,u)\nabla u\Bigr)\qquad \text{on }\partial\Omega\,.
\end{align*}
Let $\phi \in V \cap H^2(\Omega)$ be a smooth test function. Writing 
\[
    A(x) \coloneqq \K(x,u(x))^\top\in\RR^{d \times d}, \qquad b(x) \coloneqq \K_u(x,u(x))\,\nabla u(x)\in\RR^d, 
\]
a double integration-by-parts yields the adjoint identity
\[
\langle D\mathcal L(u)\,v,\phi\rangle
=
\int_\Omega (\nabla v)^\top A\,\nabla\phi\,dx
+\int_\Omega v\,b^\top\nabla\phi\,dx
+\int_{\partial \Omega} \nu\cdot(A \nabla \phi)\, v\, d\sigma - \int_{\partial \Omega} D\mathcal N(u)\,v\, \phi\, d\sigma.
\]
The boundary terms vanish for all admissible variations $v\in V \cap H^2(\Omega)$ if we impose
\[
\phi=0 \ \text{on }\Gamma_D,
\qquad
\nu\cdot(A\nabla\phi)=0 \ \text{on }\Gamma_N,
\]
and we recall that on $\Gamma_N$ the admissible variations satisfy $D\mathcal N(u)\,v=0$, as the Neumann data is fixed. Hence, for a given right-hand side $\xi$, the adjoint problem reads:
find $\phi$ such that
\begin{equation}\label{eq:adjoint-strong}
\left\{
\begin{aligned}
-\nabla\cdot(A\nabla\phi)+b\cdot\nabla\phi &= \xi &&\text{in }\Omega,\\
\phi&=0 &&\text{on }\Gamma_D,\\
\nu\cdot(A\nabla\phi)&=0 &&\text{on }\Gamma_N.
\end{aligned}
\right.
\end{equation}
The associated weak formulation is:
find $\phi\in V$ such that
\begin{equation}\label{eq:adjoint-weak}
B(\phi,v) = (\xi,v)_{L^2(\Omega)}\qquad \forall v\in V,
\end{equation}
where the bilinear form $B:V\times V\to\RR$ is defined as
\begin{equation}\label{eq:def-bilinear}
B(\phi,v)
:=\int_\Omega (\nabla v)^\top A\,\nabla\phi\,dx
+\int_\Omega v\,b^\top\nabla\phi\,dx.
\end{equation}

\begin{lem}[Well-posedness of the adjoint problem]\label{lem:adjoint-wellposed}
Assume \eqref{eq:drift-assumption}, i.e., $b \in L^\infty(\Omega)^d$. Then for every $\xi\in L^2(\Omega)$ there exists a unique $\phi\in V$ solving \eqref{eq:adjoint-weak} (equivalently \eqref{eq:adjoint-strong} in the weak sense). Moreover, $\phi\in H^2(\Omega)$ and
$\|\phi\|_{H^2(\Omega)} \le C'\|\xi\|_{L^2(\Omega)}.$
\end{lem}

\begin{proof}
Using $A \in W^{1,\infty}(\Omega)^{d\times d}$, $b \in L^\infty(\Omega)^d$ and Poincar\'e's inequality, we note that the bilinear form $B:V\times V\to\RR$ is continuous
\[
|B(\phi,v)|
\le \|A\|_{L^\infty(\Omega)^{d\times d}}\|\nabla\phi\|_{L^2(\Omega)}\|\nabla v\|_{L^2(\Omega)}
+\|b\|_{L^\infty(\Omega)^d}\|\nabla\phi\|_{L^2(\Omega)}\|v\|_{L^2(\Omega)}
\le C\|\phi\|_V\|v\|_V.
\]
Next, we split $B = B_1 + B_2$ such that
\[
B_1(\phi,v):=\int_\Omega (\nabla v)^\top A\,\nabla\phi\,dx,
\qquad
(B_2\phi)(v):=\int_\Omega v\,b^\top\nabla\phi\,dx.
\]
By uniform ellipticity of $A$, its symmetric part $(A + A^\T)/2$ is also uniformly elliptic and thus $B_1$ is coercive on $V$, $B_1(\phi,\phi)\ge \lambda\|\phi\|_V^2$. Hence the operator $T:V\to V'$ induced by $B_1$ is an isomorphism.

Since $A\in W^{1,\infty}(\Omega)^{d\times d}$, the mixed boundary problem associated with $B_1$ satisfies the hypotheses of Proposition \ref{prop:elliptic_regularity}; consequently, the solution operator $T^{-1}:V'\to V$ maps $L^2(\Omega)\subset V'$ boundedly into $H^2(\Omega)\cap V$.

Since $B_2:V\to L^2(\Omega)\hookrightarrow V'$ is bounded and the embedding $H^2(\Omega)\hookrightarrow H^1(\Omega)$ is compact, the operator $K:=T^{-1}\circ B_2:V\to V$ is compact.
Therefore, the operator $I+K:V\to V$ is Fredholm of index $0$.
Consequently, solvability of \eqref{eq:adjoint-weak} reduces to uniqueness of the homogeneous problem:
\[
-\nabla\cdot(A\nabla\phi) + b\cdot\nabla\phi = 0 \quad \text{in }\Omega,
\qquad
\phi=0 \text{ on }\Gamma_D,
\qquad
\nu\cdot(A\nabla\phi)=0 \text{ on }\Gamma_N.
\]
Trudinger's maximum principle for generalized boundary value problems \cite[Theorem~6]{Trudinger77} can be applied to the above problem (by choosing $\gamma = 1$ in their assumption (22)). The theorem then implies that any homogeneous weak solution is constant in $\Omega$. Since the space $V$ does not contain the constant function $1$ (because $|\Gamma_D|>0$),
the only possible solution is $\phi\equiv 0$.

By Fredholm index $0$ and $\ker(I+K)=\{0\}$, the operator $I+K$ is bijective; thus \eqref{eq:adjoint-weak} has a unique solution $\phi \in H^1(\Omega)$ for each $\xi\in L^2(\Omega)$. Moreover, taking $v=\phi$ in \eqref{eq:adjoint-weak} and using the uniform ellipticity of $A$,
the boundedness of $b$, and the Poincaré inequality, we obtain the estimate
\begin{equation}\label{eq:adjoint-H1-estimate}
\|\phi\|_{H^1(\Omega)} \le C \|\xi\|_{L^2(\Omega)}.
\end{equation}
It remains to prove the $H^2$-regularity. We rewrite equation \eqref{eq:adjoint-strong} solved by $\phi$ as
\[
-\nabla\cdot(A\nabla\phi)=\xi-b\cdot\nabla\phi.
\]
Because $b\in L^\infty(\Omega)^d$ and $\phi\in H^1(\Omega)$, the right-hand
side belongs to $L^2(\Omega)$. The boundary conditions remain homogeneous,
therefore, Proposition~\ref{prop:elliptic_regularity} applied to this mixed problem yields
$\phi\in H^2(\Omega)$. Moreover,
\[
\|\phi\|_{H^2(\Omega)}
\le C \|\xi-b\cdot\nabla\phi\|_{L^2(\Omega)}
\le C\bigl(\|\xi\|_{L^2(\Omega)}+\|\phi\|_{H^1(\Omega)}\bigr).
\]
Using the $H^1$-estimate obtained above, we conclude that
\[
\|\phi\|_{H^2(\Omega)} \le C\|\xi\|_{L^2(\Omega)}.
\]
\end{proof}

\begin{thm}[Optimal $H^1$ and $L^2$ error estimates]\label{thm:optimal-rates}
Assume \eqref{K1}--\eqref{K4}, \eqref{eq:estimates-assumption}, and \eqref{eq:drift-assumption} for the unique weak solution $u$ of \eqref{eq:strong_problem}. Then there exist constants $C = C\bigl(\|u\|_{H^2(\Omega)},\|g\|_{H^{3/2}(\Gamma_D)},\|b\|_{L^\infty(\Omega)^d}, \Omega, \sigma,\lambda,\Lambda,C_u\bigr) > 0$ and $h_0>0$ such that, for every $h\in(0,h_0)$ and every discrete solution $u_h$ of \eqref{eq:weak_problem_discrete}, the error $e_h:=u-u_h$ satisfies
\begin{equation}
    \|e_h\|_{L^2(\Omega)} \le C h^2, \qquad \|e_h\|_{H^1(\Omega)} \le C h.
\end{equation}
\end{thm}

\begin{proof}
Since $u_h-u_g\in V_h\subset V$ and $u-u_g\in V$, we have $e_h\in V$. By Theorem~\ref{thm:energy-reduction}, there exists a constant $C > 0$ such that
\begin{equation}\label{eq:energy-pre}
\|e_h\|_{H^1(\Omega)} \le C\bigl(h+\|e_h\|_{L^2(\Omega)}\bigr).
\end{equation}
We now estimate $\|e_h\|_{L^2(\Omega)}$ by a duality argument, using the Aubin--Nitsche trick. By Lemma~\ref{lem:adjoint-wellposed}, there exists a unique solution $\phi\in V\cap H^2(\Omega)$ of the adjoint problem \eqref{eq:adjoint-weak} with $\xi=e_h$, such that $\|\phi\|_{H^2(\Omega)}\le C\|e_h\|_{L^2(\Omega)}$ and
\begin{equation}\label{eq:adjoint-weak-proof}
B(\phi,v)=(e_h,v)_{L^2(\Omega)}
\qquad \forall v\in V.
\end{equation}
Taking $v=e_h\in V$ in \eqref{eq:adjoint-weak-proof}, and letting $\phi_h:=I_h\phi\in V_h$ be the nodal piecewise linear interpolant, we obtain
\begin{equation}\label{eq:dual-split}
\|e_h\|_{L^2(\Omega)}^2=B(\phi,e_h)=B(\phi-\phi_h,e_h)+B(\phi_h,e_h).
\end{equation}

\noindent Using the continuity of $B$, the standard interpolation estimates of \cite[Theorem 4.4.20]{BrennerScott}, and the elliptic regularity from Lemma~\ref{lem:adjoint-wellposed},
\begin{align}
|B(\phi-\phi_h,e_h)|
&\le C\|e_h\|_{H^1(\Omega)}\|\phi-\phi_h\|_{H^1(\Omega)}
\nonumber\\
&\le Ch\|e_h\|_{H^1(\Omega)}\|\phi\|_{H^2(\Omega)}
\le Ch\|e_h\|_{H^1(\Omega)}\|e_h\|_{L^2(\Omega)}.
\label{eq:first-term}
\end{align}

\noindent Regarding the second term, recall from the definition of $B$ that
\[
B(\phi_h,e_h)
=
\int_\Omega (\nabla e_h)^\top \K^\T(x,u)\nabla\phi_h\,dx
+
\int_\Omega e_h\left[(\K_u(x,u)\nabla u)^\top\nabla\phi_h\right]\,dx.
\]
Since
\[
\int_\Omega (\nabla e_h)^\top \K^\T(x,u)\nabla\phi_h\,dx
=
a(u;u,\phi_h)-a(u;u_h,\phi_h),
\]
and since $\phi_h\in V_h\subset V$, the continuous and discrete weak formulations give
\[
a(u;u,\phi_h)=\ell(\phi_h)=a(u_h;u_h,\phi_h).
\]
Therefore
\begin{align}
B(\phi_h,e_h)
&=
a(u_h;u_h,\phi_h)-a(u;u_h,\phi_h)
+
\int_\Omega e_h(\K_u(x,u)\nabla u)^\top\nabla\phi_h\,dx
\nonumber\\
&=
\int_\Omega (\nabla u_h)^\top(\K(x,u_h)-\K(x,u))\nabla\phi_h\,dx
+
\int_\Omega e_h\left[(\K_u(x,u)\nabla u)^\top\nabla\phi_h\right]\,dx.
\label{eq:Bphieh}
\end{align}

\noindent By the mean value theorem, for a.e.\ $x\in\Omega$ there exists $\theta_h(x)$ between $u(x)$ and $u_h(x)$ such that
\[
\K(x,u_h)-\K(x,u)=-\K_u(x,\theta_h)e_h.
\]
Substituting into \eqref{eq:Bphieh} yields
\[
B(\phi_h,e_h)
=
\int_\Omega e_h\Bigl[(\K_u(x,u)\nabla u-\K_u(x,\theta_h)\nabla u_h)^\top\nabla\phi_h\Bigr]dx.
\]
Writing
\[
\K_u(x,u)\nabla u-\K_u(x,\theta_h)\nabla u_h
=
\K_u(x,u)\nabla e_h
+
(\K_u(x,u)-\K_u(x,\theta_h))\nabla u_h,
\]
we obtain
\[
B(\phi_h,e_h)
=
\underbrace{\int_\Omega e_h\left[(\K_u(x,u)\nabla e_h)^\top\nabla\phi_h\right]\,dx}_{R_1}
+
\underbrace{\int_\Omega e_h\left[\bigl((\K_u(x,u)-\K_u(x,\theta_h))\nabla u_h\bigr)^\top\nabla\phi_h\right]\,dx}_{R_2}.
\]
Using the bound on $\K_u$ \eqref{K3} and H\"older's inequality, we obtain
\[
    |R_1| \le C\|e_h\|_{L^6(\Omega)}\|\nabla e_h\|_{L^2(\Omega)}\|\nabla\phi_h\|_{L^3(\Omega)} \le C\|e_h\|_{L^6(\Omega)}\|e_h\|_{H^1(\Omega)}\|\nabla\phi_h\|_{L^3(\Omega)}.
\]

\noindent Since $d\le3$, the Sobolev embedding $H^1(\Omega)\hookrightarrow L^6(\Omega)$ and
$H^2(\Omega)\hookrightarrow W^{1,6}(\Omega)\hookrightarrow W^{1,3}(\Omega)$ hold, see \cite[Theorem 4.12]{Adams}. Moreover, the
nodal piecewise linear interpolant is stable in $W^{1,p}(\Omega)$ for $1\le p\le \infty$ on conforming
shape-regular meshes (see \cite[Thm.~4.4.4]{BrennerScott}). In particular,
\begin{equation}\label{eq:interpolant-embedding}
\|e_h\|_{L^6(\Omega)}\le C\|e_h\|_{H^1(\Omega)}, \quad
\|\nabla\phi_h\|_{L^3(\Omega)}\le C\|\nabla\phi\|_{L^3(\Omega)}\le C\|\phi\|_{H^2(\Omega)},
\quad
\|\nabla\phi_h\|_{L^6(\Omega)}\le C\|\phi\|_{H^2(\Omega)}.
\end{equation}
From \eqref{eq:interpolant-embedding} and the elliptic regularity from
Lemma~\ref{lem:adjoint-wellposed}, we infer
\begin{equation}\label{eq:R1}
|R_1| \le C\|e_h\|_{H^1(\Omega)}^2\|e_h\|_{L^2(\Omega)}.
\end{equation}

\noindent The bound on $\K_{uu}$ \eqref{K3} implies the Lipschitz property of $\K_u$, namely
$\|\K_u(x,u)-\K_u(x,\theta_h)\|\le C|e|$. Hence, using H\"older again,
\begin{align}
|R_2|
&\le C\int_\Omega |e|^2|\nabla u_h||\nabla\phi_h|\,dx
\le
C\|e_h\|_{L^6(\Omega)}^2
\|\nabla u_h\|_{L^2(\Omega)}
\|\nabla\phi_h\|_{L^6(\Omega)}.
\end{align}
Making once again use of \eqref{eq:interpolant-embedding} and of the Sobolev embedding $H^1(\Omega)\hookrightarrow L^6(\Omega)$, we obtain
\[
    |R_2| \le C\|e_h\|_{H^1(\Omega)}^2\|u\|_{H^1(\Omega)}\|e_h\|_{L^2(\Omega)}.
\]
The sequence $\{u_h\}$ is bounded in $H^1(\Omega)$. Hence, for sufficiently small $h$, we obtain
\begin{equation}\label{eq:R2}
|R_2|
\le C\|e_h\|_{H^1(\Omega)}^2\|e_h\|_{L^2(\Omega)}.
\end{equation}

\noindent Substituting \eqref{eq:first-term}, \eqref{eq:R1} and \eqref{eq:R2} into \eqref{eq:dual-split}, we obtain
\[
\|e_h\|_{L^2(\Omega)}^2
\le
Ch\|e_h\|_{H^1(\Omega)}\|e_h\|_{L^2(\Omega)}
+
C\|e_h\|_{H^1(\Omega)}^2\|e_h\|_{L^2(\Omega)}.
\]
Dividing by $\|e_h\|_{L^2(\Omega)}$ (the case $\|e_h\|_{L^2(\Omega)}=0$ being trivial), we infer
\[
\|e_h\|_{L^2(\Omega)}
\le
Ch\|e_h\|_{H^1(\Omega)}
+
C\|e_h\|_{H^1(\Omega)}^2.
\]
Using now the preliminary estimate \eqref{eq:energy-pre}, we obtain
\[
\|e_h\|_{L^2(\Omega)}
\le
Ch\bigl(h+\|e_h\|_{L^2(\Omega)}\bigr)
+
C\bigl(h+\|e_h\|_{L^2(\Omega)}\bigr)^2.
\]
Hence
\[
\|e_h\|_{L^2(\Omega)}
\le
C\Bigl(
h^2+h\|e_h\|_{L^2(\Omega)}+\|e_h\|_{L^2(\Omega)}^2
\Bigr).
\]
Choose $h_1>0$ such that $Ch\le \frac14$ for all $h\in(0,h_1)$. From Remark \ref{rem:strong-convergence}, $u_h \to u$ in $H^1(\Omega)$, and thus we have
$\|e_h\|_{L^2(\Omega)} \to 0$. Therefore, there exists $h_2>0$ such that
$C\|e_h\|_{L^2(\Omega)}\le \frac14$ for all $h\in(0,h_2)$. Finally, for
$0 < h < h_0 \coloneqq \min\{h_1,h_2\}$,
\[
\|e_h\|_{L^2(\Omega)}
\le
Ch^2+\frac14\|e_h\|_{L^2(\Omega)}+\frac14\|e_h\|_{L^2(\Omega)},
\]
whence
\[
\|e_h\|_{L^2(\Omega)}\le Ch^2.
\]
Substituting this bound into \eqref{eq:energy-pre}, we also obtain $\|e_h\|_{H^1(\Omega)}\le Ch$.
\end{proof}

\section{Newton convergence for the discrete Galerkin system and transfer to the exact solution}\label{seq:NK}

In this section, we show that for each fixed mesh size $h$, the nonlinear algebraic system associated with the Galerkin discretization \eqref{eq:weak_problem_discrete} can be solved by Newton's method under a local Newton--Kantorovich assumption at the initial iterate. The resulting Newton iterates converge to a discrete Galerkin solution, and the algebraic error can be combined with the optimal finite element error estimates from Theorem~\ref{thm:optimal-rates} to obtain convergence to the exact weak solution of \eqref{eq:weak_problem}.

Let $\{\varphi_i\}_{i=1}^{N_h}$ be the standard Lagrange basis of $V_h$, where $N_h$ is the number of free nodes in $\mathcal{T}_h$. Then any $u_h \in u_{g,h}+V_h$ can be expressed in terms of its nodal values $\{U_j\}_{j=1}^{N_h}$ as
\[
u_h = u_{g,h} + \sum_{j=1}^{N_h} U_j \varphi_j ,
\qquad U=(U_1,\dots,U_{N_h})^\top \in \mathbb{R}^{N_h}.
\]
We define the discrete residual mapping $F_h:\mathbb{R}^{N_h}\to\mathbb{R}^{N_h}$ by
\[
(F_h(U))_i
:=
a(u_h;u_h,\varphi_i)-\ell(\varphi_i),
\qquad i=1,\dots,N_h.
\]
Then $F_h(U)=0$ if and only if the corresponding function $u_h\in u_{g,h}+V_h$ is a discrete solution of \eqref{eq:weak_problem_discrete}.

Since $U \mapsto u_h(U)$ is affine and $F_h(U)$ is simply the
continuous residual evaluated at $u_h(U)$ and tested against $\varphi_i$,
the differentiation argument used for the continuous operator applies
identically. Consequently, the mapping $F_h$ is continuously differentiable
on $\mathbb R^{N_h}$. More precisely, for any $U, V\in\mathbb R^{N_h}$ one has
\[
(F_h'(U)V)_i
=
\int_\Omega
\bigl(\K(x,u_h(U))\nabla \delta u_h(V)\bigr)\cdot\nabla\varphi_i\,dx
+
\int_\Omega
\bigl(\K_u(x,u_h(U))\,\delta u_h(V)\,\nabla u_h(U)\bigr)
\cdot\nabla\varphi_i\,dx,
\]
for $i=1,\dots,N_h$, where
\[
\delta u_h(V):=\sum_{j=1}^{N_h}V_j\varphi_j\in V_h.
\]

In particular, the Jacobian matrix $J_h(U)=F_h'(U)\in\mathbb R^{N_h\times N_h}$
has the entries
\[
(J_h(U))_{ij}
=
\int_\Omega
\bigl(\K(x,u_h(U))\nabla\varphi_j\bigr)\cdot\nabla\varphi_i\,dx
+
\int_\Omega
\bigl(\K_u(x,u_h(U))\,\varphi_j\,\nabla u_h(U)\bigr)
\cdot\nabla\varphi_i\,dx .
\]
This expression shows that the Jacobian matrix $J_h(U)$
corresponds to the Galerkin discretization of the linearized
operator $D\mathcal{L}(u_h)$.

\begin{lem}\label{lem:Fhprime-lipschitz}
The Jacobian $F_h'$ is locally Lipschitz on $\mathbb R^{N_h}$, that is,
for every bounded set $B\subset\mathbb R^{N_h}$ there exists a constant
$L_B>0$ such that
\[
\opnorm{F_h'(U)-F_h'(V)}
\le
L_B \|U-V\|_2,
\qquad \forall U,V\in B .
\]
Consequently, $F_h\in C^1(\mathbb R^{N_h};\mathbb R^{N_h})$.
\end{lem}

\begin{proof}
Let $B\subset\mathbb R^{N_h}$ be bounded and let $U,V\in B$ with the
corresponding finite element functions
$u_h(U),u_h(V)\in u_{g, h} + V_h$.
For $Z\in\mathbb R^{N_h}$ define
\[
\delta u_h(Z):=\sum_{j=1}^{N_h} Z_j\varphi_j .
\]
Using the expression of the Jacobian derived above,
\begin{align}
((F_h'(U)-F_h'(V))Z)_i
={}&
\int_\Omega
(\K(x,u_h(U))-\K(x,u_h(V)))
\nabla\delta u_h(Z)\cdot\nabla\varphi_i\,dx \label{eq:Lipschitz_integral_1}
\\
&+
\int_\Omega
(\K_u(x,u_h(U))-\K_u(x,u_h(V)))
\delta u_h(Z)\,\nabla u_h(U)\cdot\nabla\varphi_i\,dx \label{eq:Lipschitz_integral_2}
\\
&+
\int_\Omega
\K_u(x,u_h(V))\,\delta u_h(Z)\,
\nabla(u_h(U)-u_h(V))\cdot\nabla\varphi_i\,dx . \label{eq:Lipschitz_integral_3}
\end{align}
Since $\K$ is $C^2$ in the second variable \eqref{K3}, the mean value theorem
implies that for the constants
\[
M_1(B):=\sup_{x\in\Omega,\ s\in I_B}\|\K_u(x,s)\|,
\qquad
M_2(B):=\sup_{x\in\Omega,\ s\in I_B}\|\K_{uu}(x,s)\|,
\]
where $I_B$ contains the range of $u_h(W)$ for $W\in B$, one has
\[
\|\K(\cdot,u_h(U))-\K(\cdot,u_h(V))\|_{L^\infty(\Omega)}
\le
M_1(B)\|u_h(U)-u_h(V)\|_{L^\infty(\Omega)},
\]
and
\[
\|\K_u(\cdot,u_h(U))-\K_u(\cdot,u_h(V))\|_{L^\infty(\Omega)}
\le
M_2(B)\|u_h(U)-u_h(V)\|_{L^\infty(\Omega)}.
\]
Because $V_h$ is the space of $P_1$ Lagrange finite elements,
\[
\|u_h(U)-u_h(V)\|_{L^\infty(\Omega)}
=
\max_i |U_i-V_i|
\le
\|U-V\|_2,
\]
and
\[
\|\delta u_h(Z)\|_{L^\infty(\Omega)}
\le
\|Z\|_2 .
\]
Moreover, Lemma \ref{lem:grad-nodal-bound} yields
\[
\|\nabla \delta u_h(Z)\|_{L^2(\Omega)}
\le
C_\blacktriangle h^{\frac d2-1}\|Z\|_2,
\qquad
\|\nabla (u_h(U)-u_h(V))\|_{L^2(\Omega)}
\le
C_\blacktriangle h^{\frac d2-1}\|U-V\|_2,
\]
where $C_\blacktriangle = \left( \frac{(d+1)\sigma^2 N_*}{d!} \right)^{1/2}$. Since $B$ is bounded, there exists $R_B>0$ such that
$\|U\|_2\le R_B$ for all $U\in B$. Moreover, since $u_h(U)-u_{g,h}\in V_h$, Lemma~\ref{lem:grad-nodal-bound} yields
\[
\|\nabla u_h(U)\|_{L^2(\Omega)}
\le
\|\nabla u_{g,h}\|_{L^2(\Omega)}
+
\|\nabla (u_h(U)-u_{g,h})\|_{L^2(\Omega)}
\le
\|\nabla u_{g,h}\|_{L^2(\Omega)}
+
C_\blacktriangle h^{\frac d2-1}R_B.
\]
Applying the Cauchy--Schwarz inequality to each of the three terms
\eqref{eq:Lipschitz_integral_1}--\eqref{eq:Lipschitz_integral_3},
then summing the resulting bounds over $i=1,\dots,N_h$ and using the
uniformly bounded overlap of the supports of the basis functions, we obtain
\[
\|(F_h'(U)-F_h'(V))Z\|_2
\le
C_\blacktriangle h^{\frac d2-1}
\Bigl(
2C_\blacktriangle h^{\frac d2-1}M_1(B)
+
M_2(B)\bigl(\|\nabla u_{g,h}\|_{L^2(\Omega)}+C_\blacktriangle h^{\frac d2-1}R_B\bigr)
\Bigr)
\|U-V\|_2\|Z\|_2 .
\]
Hence
\[
\opnorm{F_h'(U)-F_h'(V)}
\le
L_B\|U-V\|_2,
\]
where
\[
L_B
=
2C_\blacktriangle^2 h^{d-2}M_1(B)
+
C_\blacktriangle h^{\frac d2-1}M_2(B)
\bigl(\|\nabla u_{g,h}\|_{L^2(\Omega)}+C_\blacktriangle h^{\frac d2-1}R_B\bigr).
\]
Thus $F_h'$ is locally Lipschitz, and in particular continuous.
Combined with the differentiability established above,
this implies that $F_h \in C^1(\mathbb{R}^{N_h};\mathbb{R}^{N_h})$.
\end{proof}

Next, to solve the nonlinear finite element system $F_h(U) = 0$, for a given initial vector $U^{(0)}\in\mathbb{R}^{N_h}$, we consider the Newton iteration
\begin{equation}\label{eq:Newton-iteration}
    U^{(k+1)} = U^{(k)} - \bigl(F_h'(U^{(k)})\bigr)^{-1}F_h(U^{(k)}), \qquad k\ge 0,
\end{equation}
whenever $F_h'(U^{(k)})$ is invertible. The associated finite element iterates corresponding to the nodal values $\{U_j^{(k)}\}_{j=1}^{N_h}$ are
\begin{equation}\label{eq:FEM-iteration}
    u_h^{(k)} := u_{g,h}+\sum_{j=1}^{N_h} U_j^{(k)} \varphi_j .
\end{equation}

\begin{thm}[Newton--Galerkin convergence]\label{thm:newton-galerkin}
Under the hypotheses of Theorem \ref{thm:optimal-rates}, fix $h>0$ and let $U^{(0)}\in\mathbb R^{N_h}$ such that $F_h'(U^{(0)})$ is invertible. Assume that there exists $r_h>0$ such that
\begin{enumerate}
\item[(i)] $\|F_h'(U^{(0)})^{-1}F_h(U^{(0)})\|_2 \le r_h/2$,
\item[(ii)] for all $U,V\in\overline{B}(U^{(0)},r_h)$, the Jacobian satisfies the Lipschitz condition
\[
\|F_h'(U^{(0)})^{-1}(F_h'(U)-F_h'(V))\|_{\mathcal L(\mathbb R^{N_h})}
\le
\frac{1}{r_h}\|U-V\|_2 .
\]
\end{enumerate}
Then there exists a unique vector $U_h^\ast\in\overline{B}(U^{(0)},r_h)$
such that $F_h(U_h^\ast)=0$, and the Newton iterates \eqref{eq:Newton-iteration} are well-defined and stay in the ball $\overline{B}(U^{(0)},r_h)$, converging to $U_h^\ast$ such that 
\[
\|U^{(k)}-U_h^\ast\|_2 \le \frac{r_h}{2^k}, \qquad k\ge0.
\]
Moreover, if $k=k_h$ is chosen such that
\begin{equation}\label{eq:NK_stopping_criterion}
\frac{r_h}{2^{k_h}}\le\, h^{2-\frac d2},
\end{equation}
then the optimal finite element estimates hold:
\[
\|u-u_h^{(k_h)}\|_{H^1(\Omega)}\le C_1 h,
\qquad
\|u-u_h^{(k_h)}\|_{L^2(\Omega)}\le C_2 h^2\,,
\]
where $u$ denotes the exact unique weak solution of \eqref{eq:weak_problem}. This shows that the algebraic error introduced by the Newton iteration
does not deteriorate the optimal finite element convergence rates,
provided a sufficient number of iterations is performed.
\end{thm}

\begin{proof}
By Lemma~\ref{lem:Fhprime-lipschitz}, the mapping $F_h$ is of class $C^1$.
Hence the one-constant Newton--Kantorovich theorem \cite[Theorem 5]{CiarletMardare12}
applies, yielding the existence of a unique zero
$U_h^\ast\in\overline{B}(U^{(0)},r_h)$ and the estimate
\begin{equation}\label{eq:Newton-estimates}
    \|U^{(k)}-U_h^\ast\|_2 \le \frac{r_h}{2^k},\qquad k\ge 0.
\end{equation}
In particular, all iterates $U^{(k)}$ remain in $\overline{B}(U^{(0)},r_h)$,
so the assumptions are valid at each step. The relation $F_h(U_h^\ast)=0$ implies that the corresponding function
$u_h^\ast$ is a discrete solution of \eqref{eq:weak_problem_discrete}. Set $W^{(k)}:=U^{(k)}-U_h^\ast$ and observe that
\[
u_h^{(k)}-u_h^\ast
=
\sum_{j=1}^{N_h} W_j^{(k)}\varphi_j .
\]
By Lemma \ref{lem:grad-nodal-bound} and Poincar\'e, there exist two constants $\overline{C}_1, \overline{C}_2>0$,
independent of $h$ and $k$, such that
\[
    \|u_h^{(k)}-u_h^\ast\|_{H^1(\Omega)} \le C_P \|\nabla (u_h^{(k)}-u_h^\ast) \|_{L^2(\Omega)} \le \overline{C}_1 h^{\frac d2-1}\|W^{(k)}\|_2, \qquad
    \|u_h^{(k)}-u_h^\ast\|_{L^2(\Omega)}
    \le
    \overline{C}_2 h^{\frac d2}\|W^{(k)}\|_2.
\]
Combining this with the Newton estimate \eqref{eq:Newton-estimates} gives
\[
\|u_h^{(k)}-u_h^\ast\|_{H^1(\Omega)}
\le
\overline{C}_1 h^{\frac d2-1}\frac{r_h}{2^k},
\qquad
\|u_h^{(k)}-u_h^\ast\|_{L^2(\Omega)}
\le
\overline{C}_2 h^{\frac d2}\frac{r_h}{2^k}.
\]
On the other hand, since $u_h^\ast$ is a discrete solution, Theorem~\ref{thm:optimal-rates} implies that there exists $\widetilde{C}_1, \widetilde{C}_2 > 0$
\[
\|u-u_h^\ast\|_{H^1(\Omega)}\le \widetilde{C}_1 h,
\qquad
\|u-u_h^\ast\|_{L^2(\Omega)}\le \widetilde{C}_2 h^2.
\]
Therefore, by the triangle inequality,
\[
\|u-u_h^{(k)}\|_{H^1(\Omega)}
\le
\|u-u_h^\ast\|_{H^1(\Omega)}
+
\|u_h^\ast-u_h^{(k)}\|_{H^1(\Omega)}
\le
\widetilde{C}_1 h + \overline{C}_1 h^{\frac d2-1}\frac{r_h}{2^k},
\]
and similarly,
\[
\|u-u_h^{(k)}\|_{L^2(\Omega)}
\le
\|u-u_h^\ast\|_{L^2(\Omega)}
+
\|u_h^\ast-u_h^{(k)}\|_{L^2(\Omega)}
\le
\widetilde{C}_2 h^2 + \overline{C}_2 h^{\frac d2}\frac{r_h}{2^k}.
\]
Finally, if $k=k_h$ is chosen so that
\[
\frac{r_h}{2^{k_h}}\le \,h^{2-\frac d2},
\]
then there exist constants $C_1,C_2>0$ such that
\[
\|u-u_h^{(k_h)}\|_{H^1(\Omega)}\le C_1 h,
\qquad
\|u-u_h^{(k_h)}\|_{L^2(\Omega)}\le C_2 h^2.
\]
\end{proof}
\begin{rem}\label{rem:practical_NK}
The explicit Lipschitz bound from Lemma~\ref{lem:Fhprime-lipschitz}
allows the hypotheses of Theorem~\ref{thm:newton-galerkin} to be
verified in practice. For an initial guess $U^{(0)}$ such that $F_h'(U^{(0)})$ is invertible, let
\[
\delta U^{(0)}:=-\bigl(F_h'(U^{(0)})\bigr)^{-1}F_h(U^{(0)}),
\qquad
\beta_h:=\|\delta U^{(0)}\|_2,
\]
and define the candidate radius $r_h := 2\beta_h$. Then assumption~(i) of Theorem~\ref{thm:newton-galerkin} is automatically satisfied. For the ball $B:=B(U^{(0)},r_h)$, we have
\[
\|U\|_2 \le R_B := \|U^{(0)}\|_2 + r_h,
\qquad \forall U\in B.
\]
Hence, Lemma~\ref{lem:Fhprime-lipschitz} yields the Lipschitz bound
\[
L_B
=
2C_\blacktriangle^2 h^{d-2}M_1(B)
+
C_\blacktriangle h^{\frac d2-1}M_2(B)
\bigl(\|\nabla u_{g,h}\|_{L^2(\Omega)}+C_\blacktriangle h^{\frac d2-1}R_B\bigr)\,,
\]
where $M_1(B)$ and $M_2(B)$ are evaluated on any interval containing the range of $u_h(W)$ for $W\in B$. Since
\[
\|u_h(W)\|_{L^\infty(\Omega)}
\le
\|u_{g,h}\|_{L^\infty(\Omega)}+\|W\|_2
\le
\|u_{g,h}\|_{L^\infty(\Omega)}+R_B
=
\|g\|_{L^\infty(\Gamma_D)} + R_B\,,
\]
for $u_{g, h}$ chosen as in Remark~\ref{rem:lifting}, it suffices to take
\[
M_1(B):=
\sup_{\substack{x\in\Omega\\ |s|\le \|g\|_{L^\infty(\Gamma_D)} + R_B}}
\|\K_u(x,s)\|,
\qquad
M_2(B):=
\sup_{\substack{x\in\Omega\\ |s|\le \|g\|_{L^\infty(\Gamma_D)} + R_B}}
\|\K_{uu}(x,s)\|.
\]
Assumption~(ii) of Theorem~\ref{thm:newton-galerkin} is then ensured if
\begin{equation}\label{eq:practical-Lpishitz-bound}
\bigl\|\bigl(F_h'(U^{(0)})\bigr)^{-1}\bigr\|_{\mathcal L(\mathbb R^{N_h})}
\, L_B
=
\frac{L_B}{\sigma_{\min}(F_h'(U^{(0)}))}
\le \frac{1}{r_h}\,, 
\end{equation}
where $\sigma_{\min}(F_h'(U^{(0)})) > 0$ is the smallest singular value of the invertible matrix $F_h'(U^{(0)})$.












\end{rem}

\clearpage

\section{Numerical results}

The purpose of this section is to demonstrate that Theorem~\ref{thm:newton-galerkin} is algorithmically applicable through the practical criterion established in Remark~\ref{rem:practical_NK}. Through representative two- and three-dimensional examples, we show that its hypotheses can be verified a posteriori from computable quantities associated with the discrete problem, thereby providing a practical certificate for the convergence of Newton's method. We then verify numerically the optimal finite element error estimates guaranteed by the theorem.

\subsection{Examples}

We present two representative numerical examples illustrating the performance of the proposed method. The first example is set in a two-dimensional doubly connected domain, while the second example extends the framework to a three-dimensional setting, see Figures~\ref{fig:mesh_2D} and~\ref{fig:mesh_3D}, respectively.

\begin{figure}[H]
\centering
\subfigure[Example $1$.\label{fig:mesh_2D}]{%
    \includegraphics[width=0.45\textwidth]{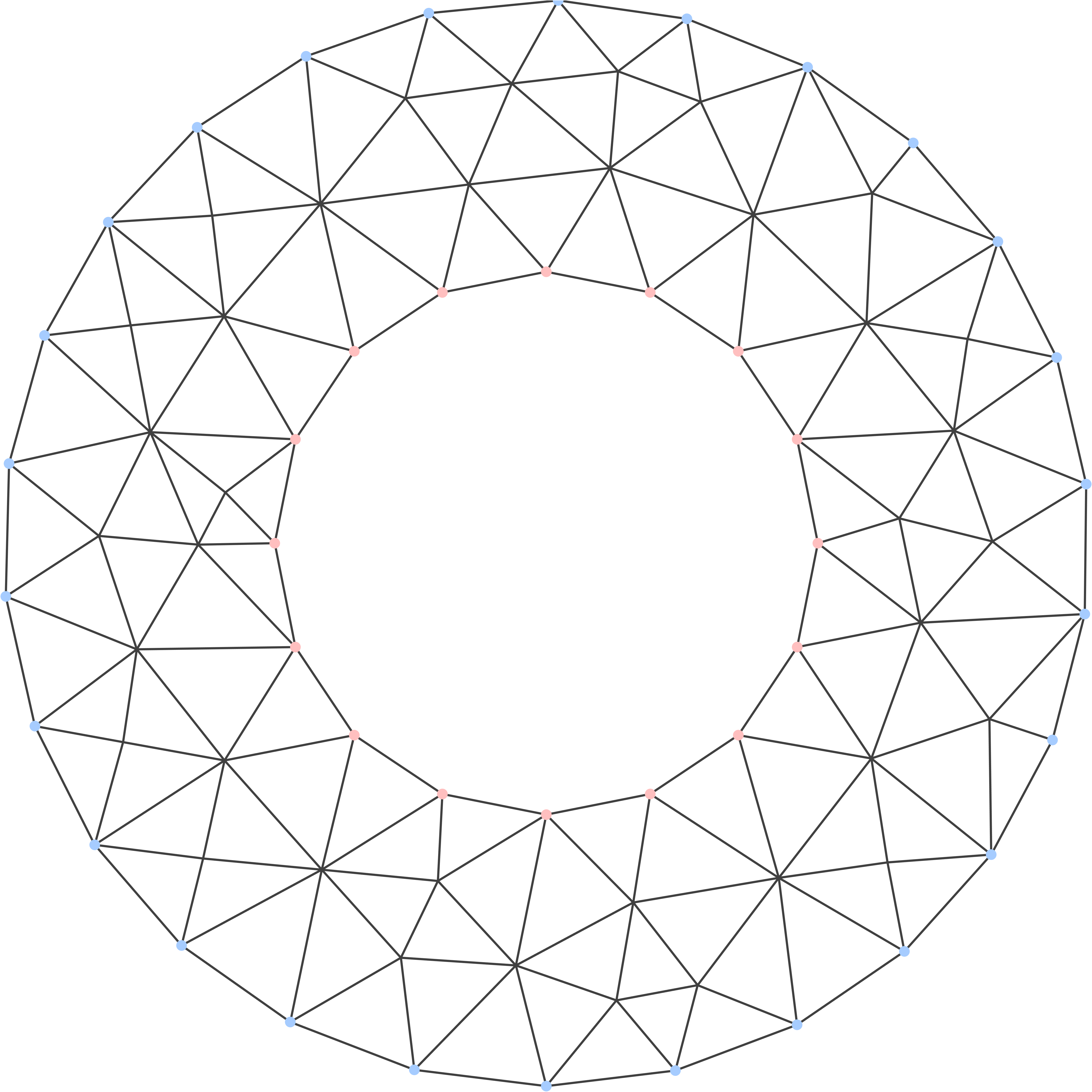}%
}
\hfill
\subfigure[Example $2$.\label{fig:mesh_3D}]{%
    \includegraphics[width=0.45\textwidth]{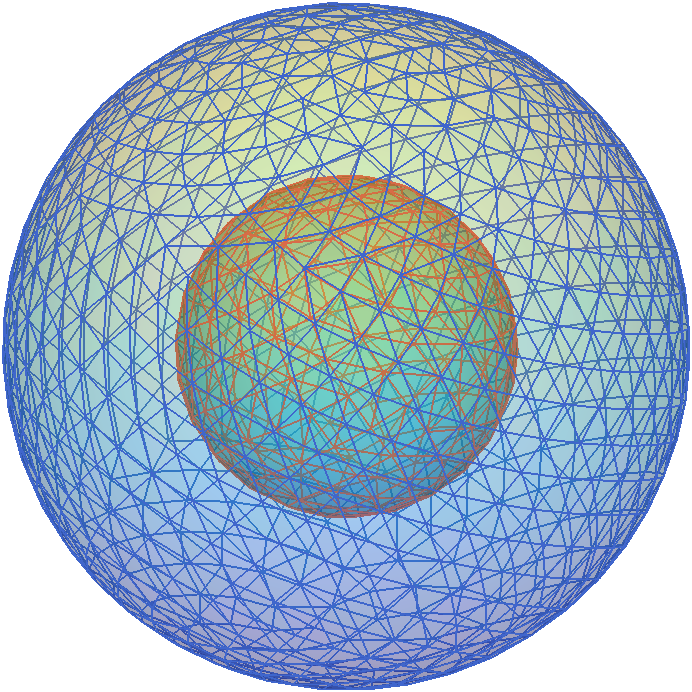}%
}
\caption{Finite element meshes for the computational domains used in the numerical experiments.}
\label{fig:meshes}
\end{figure}

\subsubsection*{Example 1 (2D case)}
Consider a solid occupying the two-dimensional annular domain $\Omega \coloneqq \mathrm{B}(\mathrm{O},r_\mathrm{out}) \setminus \overline{\mathrm{B}(\mathrm{O},r_\mathrm{int})}$, where $\mathrm{O}(0,0)$ is the origin of the Cartesian coordinate system, $r_\mathrm{int} = 0.5$ and $r_\mathrm{out} = 1$ are the inner and outer radii of $\Omega$, respectively, characterised by the nonlinear anisotropic nonsymmetric heat conduction tensor given by
\begin{subequations} \label{eq:ex1a}
\begin{equation}\label{eq:K-example-2d} \K(x_1,x_2,u)= \begin{pmatrix} 2+(x_1-0.4)^2+0.45\tanh(u) & 0.35x_1x_2+0.30\sin(u) \\[2mm] 0.05x_1x_2+0.22\arctan(u) & 1.8+(x_2-0.9)^2+0.40\dfrac{u}{\sqrt{1+u^2}} \end{pmatrix}. 
\end{equation}
Neumann boundary conditions are imposed on the inner boundary $\Gamma_N \coloneqq \{(x_1,x_2)\in\mathbb{R}^2 : \sqrt{x_1^2 + x_2^2} = r_\mathrm{int}\}$, while Dirichlet boundary conditions are prescribed on the outer boundary $\Gamma_D \coloneqq \{(x_1,x_2)\in\mathbb{R}^2 : \sqrt{x_1^2 + x_2^2} = r_\mathrm{out}\}$, such that $\Gamma_D \cup \Gamma_N = \partial \Omega$ and $\overline{\Gamma}_D \cap \overline{\Gamma}_N = \varnothing$, with the corresponding manufactured analytical solution for the temperature given by
\begin{align} \label{eq:ex1b}
u^{(\mathrm{ex})}(x_1,x_2) = \cos{x_1} \cosh{x_2} + \sin{x_1} \sinh{x_2},
\quad (x_1,x_2) \in \overline{\Omega}.
\end{align}
\end{subequations}

\subsubsection*{Example 2 (3D case)}
Consider now a solid occupying the three-dimensional spherical shell domain
$\Omega \coloneqq \mathrm{B}(\mathrm{O},r_\mathrm{out}) \setminus \overline{\mathrm{B}(\mathrm{O},r_\mathrm{int})}$,
where $r_\mathrm{int} = 0.5$ and $r_\mathrm{out} = 1$ are the inner and outer radii of $\Omega$, respectively,
characterised by the nonlinear anisotropic heat conduction tensor given by
\begin{subequations} \label{eq:ex2}
\begin{equation} \label{eq:ex2a}
\K(x_1,x_2,x_3, u) =
(1 + 0.5 \tanh(u))
\begin{bmatrix}
(x_1 - 0.4)^2 + 2 & 2 x_1 x_2 & 2 x_1 x_3 \\
2 x_1 x_2 & (x_2 - 0.9)^2 + 1 & 2 x_2 x_3 \\
2 x_1 x_3 & 2 x_2 x_3 & (x_3 - 0.7)^2 + 1
\end{bmatrix},
\quad (x_1,x_2,x_3) \in \overline{\Omega}.
\end{equation}
Similarly to Example 1, Neumann boundary conditions are imposed on the inner boundary,
whilst Dirichlet boundary conditions are prescribed on the outer boundary, whereas the corresponding manufactured analytical solution for the temperature is given by
\begin{align} \label{eq:ex2b}
u^{(\mathrm{ex})}(x_1,x_2,x_3)
= \cos(x_1)\cosh(x_3) + \sin(x_2)\sinh(x_3) + \cos(x_2)\sinh(x_1),
\quad (x_1,x_2,x_3) \in \overline{\Omega}.
\end{align}
\end{subequations}

In both examples, the source term $f$ and the boundary data $g$ and $h$ are chosen consistently with the manufactured solution $u^{(\mathrm{ex})}$, by evaluating the differential operator and the prescribed boundary operators in the governing problem. This ensures that $u^{(\mathrm{ex})}$ is the exact solution of the boundary value problem and enables an exact assessment of the finite element discretization errors.

\subsection{A posteriori Newton--Kantorovich condition}

A crucial component in the practical realization of the Newton--Galerkin scheme is the construction of a reliable initial approximation for the nonlinear discrete problem. To this end, we first solve a frozen-coefficient problem, in which the nonlinear diffusion tensor $\K(x,u)$ is evaluated at the reference state $u=0$. This leads to the linear elliptic boundary value problem
\begin{equation}\label{eq:frozen-coeff}
\left\{
\begin{aligned}
-\nabla\cdot\bigl(\K(x, 0)\nabla u\bigr) &= f && \text{in }\Omega,\\
u &= g && \text{on }\Gamma_D,\\
\nu\cdot\bigl(\K(x,0)\nabla u\bigr) &= h && \text{on }\Gamma_N.
\end{aligned}
\right.
\end{equation}

Tables~\ref{tab:ex1} and~\ref{tab:ex2} summarize the practical verification of the Newton--Kantorovich criterion for Examples~1 and~2, respectively. Starting from the finite element solution of the frozen-coefficient problem, \(m\) preliminary Newton iterations are performed, and the first iterate for which the computable condition~\eqref{eq:practical-Lpishitz-bound} is satisfied is designated by \(U^{(0)}\). Thus, \(U^{(0)}\) is not the frozen-coefficient solution, but the first Newton iterate satisfying the certification condition. For each mesh size, the tables report the corresponding certification quantities, including \(r_h\), the radius of the ball centred at \(U^{(0)}\) within which the associated discrete solution is guaranteed to be unique. In particular, the inequality~\eqref{eq:practical-Lpishitz-bound} 
\[
    \frac{L_B}{\sigma_{\min}(F_h'(U^{(0)}))}\leq\frac{1}{r_h}
\]
is satisfied at every discretization level considered, thereby providing an \emph{a posteriori} verification of assumption~(ii) of Theorem~\ref{thm:newton-galerkin}; see Remark~\ref{rem:practical_NK}.

\begin{table}[H]
\centering
\begin{tabular}{c c c c c c c}
\toprule
$h$ & $m$ & $r_h$ & $\sigma_{\min}(F_h'(U^{(0)}))$ & $L_B$ & $L_B/\sigma_{\min}(F_h'(U^{(0)}))$ & $1/r_h$ \\
\midrule
0.11 & 2 & $4.16\times 10^{-9}$ & $1.51\times 10^{-1}$ & $2.13\times 10^{4}$ & $1.41\times 10^{5}$ & $2.41\times 10^{8}$ \\
0.06 & 2 & $7.95\times 10^{-9}$ & $4.20\times 10^{-2}$ & $3.77\times 10^{4}$ & $8.97\times 10^{5}$ & $1.26\times 10^{8}$ \\
0.03 & 2 & $1.37\times 10^{-8}$ & $1.25\times 10^{-2}$ & $9.37\times 10^{4}$ & $7.48\times 10^{6}$ & $7.31\times 10^{7}$ \\
\bottomrule
\end{tabular}
\caption{Values of the mesh size $h$, the number of preliminary Newton iterations $m$, the corresponding Newton radius $r_h=2\|U^{(1)}-U^{(0)}\|_2$, the smallest singular value of the Jacobian matrix $\sigma_{\min}(F_h'(U^{(0)}))$, the Lipschitz bound $L_B$, and the ratios $L_B/\sigma_{\min}(F_h'(U^{(0)}))$ and $1/r_h$, for Example $1$.}
\label{tab:ex1}
\end{table}

\begin{table}[H]
\centering
\small
\begin{tabular}{c c c c c c c}
\toprule
$h$ & $m$ & $r_h$ & $\sigma_{\min}(F_h'(U^{(0)}))$ & $L_B$ & $L_B/\sigma_{\min}(F_h'(U^{(0)}))$ & $1/r_h$ \\
\midrule
0.32 & 3 & $2.58\times 10^{-14}$ & $4.76\times 10^{-2}$ & $5.86\times 10^{5}$ & $1.23\times 10^{7}$ & $3.88\times 10^{13}$ \\
0.16 & 3 & $1.62\times 10^{-13}$ & $1.20\times 10^{-2}$ & $1.89\times 10^{5}$ & $1.57\times 10^{7}$ & $6.17\times 10^{12}$ \\
0.08 & 3 & $4.43\times 10^{-13}$ & $3.02\times 10^{-3}$ & $8.91\times 10^{4}$ & $2.95\times 10^{7}$ & $2.26\times 10^{12}$ \\
\bottomrule
\end{tabular}
\caption{Values of the mesh size $h$, the number of preliminary Newton iterations $m$, the corresponding Newton radius $r_h=2\|U^{(1)}-U^{(0)}\|_2$, the smallest singular value of the Jacobian matrix $\sigma_{\min}(F_h'(U^{(0)}))$, the Lipschitz bound $L_B$, and the ratios $L_B/\sigma_{\min}(F_h'(U^{(0)}))$ and $1/r_h$, for Example $2$.}
\label{tab:ex2}
\end{table}

\subsection{Optimal error estimates}

To investigate the convergence of the finite element method, we analyse the \(H^1(\Omega)\)- and \(L^2(\Omega)\)-norms of the error in the numerical approximation, $\|u-u_h^{(k_h)}\|_{H^1(\Omega)}$ and $\|u-u_h^{(k_h)}\|_{L^2(\Omega)}$, respectively, for three different meshes in each of the two examples considered. Here \(u\) denotes the exact weak solution of \eqref{eq:weak_problem}, while $u_h^{(k_h)} = u_{g, h} + \sum_{j=1}^{N_h} U_j^{(k_h)} \varphi_j$ is the finite element iterate corresponding to the nodal vector \(\{U_j^{(k_h)}\}_{j=1}^{N_h}\), as defined in \eqref{eq:FEM-iteration}. For each mesh, the number \(k_h\) of Newton iterations is chosen so that the stopping condition \eqref{eq:NK_stopping_criterion} is satisfied. This condition is precisely the requirement used in the proof of the optimal error estimates: it ensures that the algebraic error produced by terminating the nonlinear solver is of higher order than the finite element discretization error. Thus, the computed approximation \(u_h^{(k_h)}\) is sufficiently close to the corresponding nonlinear Galerkin solution for the asymptotic finite element rates to be observed.

Figures~\ref{fig:convergence_2D} and~\ref{fig:convergence_3D} display, on a logarithmic scale, the two finite element errors $\|u-u_h^{(k_h)}\|_{H^1(\Omega)}$ and $\|u-u_h^{(k_h)}\|_{L^2(\Omega)}$, for Examples~1 and~2, respectively, as functions of the mesh size $h$, together with the reference functions $h$ and $h^2$. It can be seen from these figures that the corresponding curves are parallel, indicating first-order convergence in the $H^1(\Omega)$-norm and second-order convergence in the $L^2(\Omega)$-norm. This is consistent with the finite element error estimates established in Theorem~\ref{thm:newton-galerkin}. In particular, the theoretical convergence orders are recovered using the stopping strategy~\eqref{eq:NK_stopping_criterion}, showing that the algebraic error does not influence the asymptotic finite element convergence.

\begin{figure}[H]
\centering
\subfigure[Example~$1$.\label{fig:convergence_2D}]{%
    \includegraphics[width=0.45\textwidth]{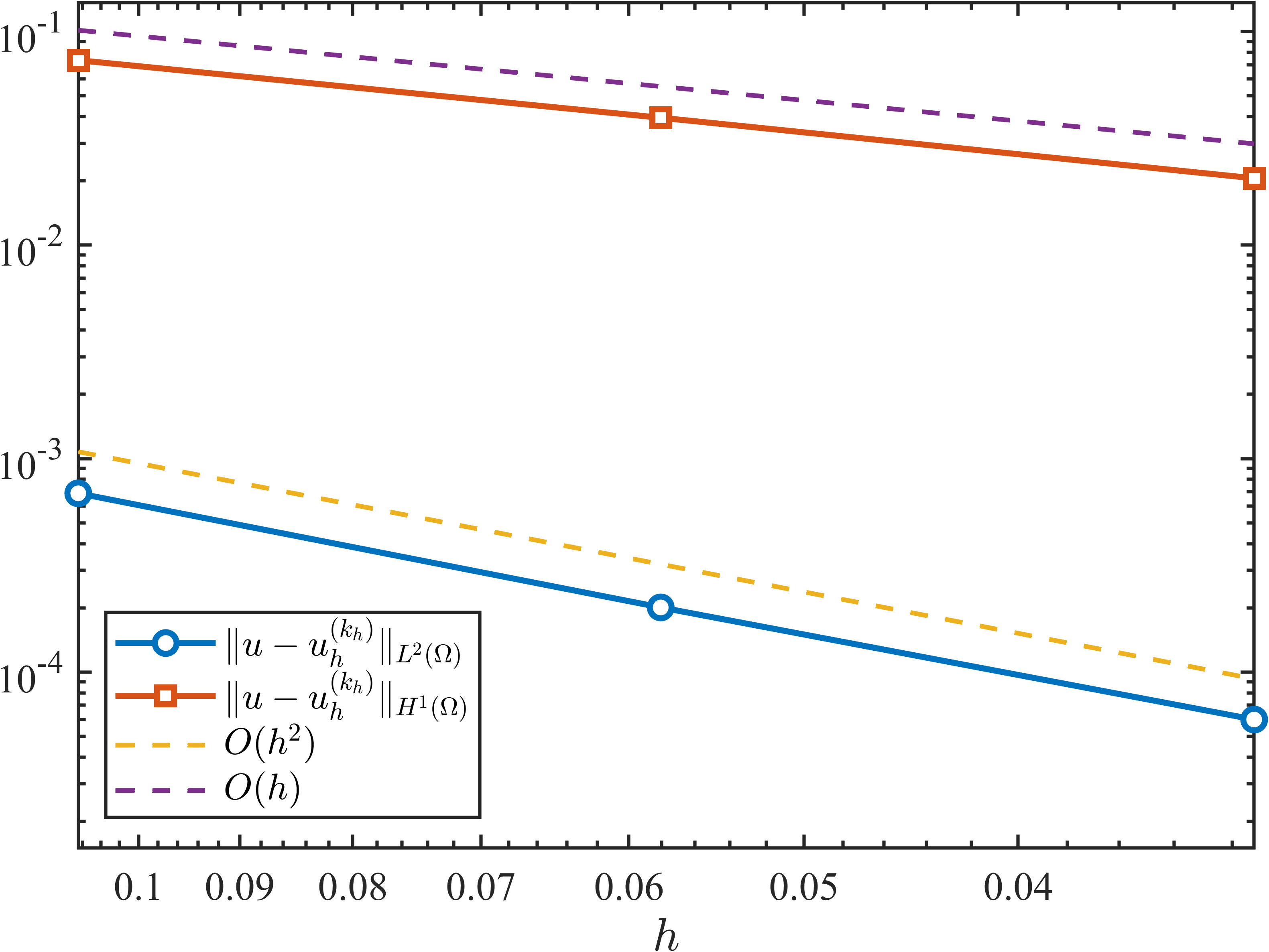}%
}
\hfill
\subfigure[Example~$2$.\label{fig:convergence_3D}]{%
    \includegraphics[width=0.45\textwidth]{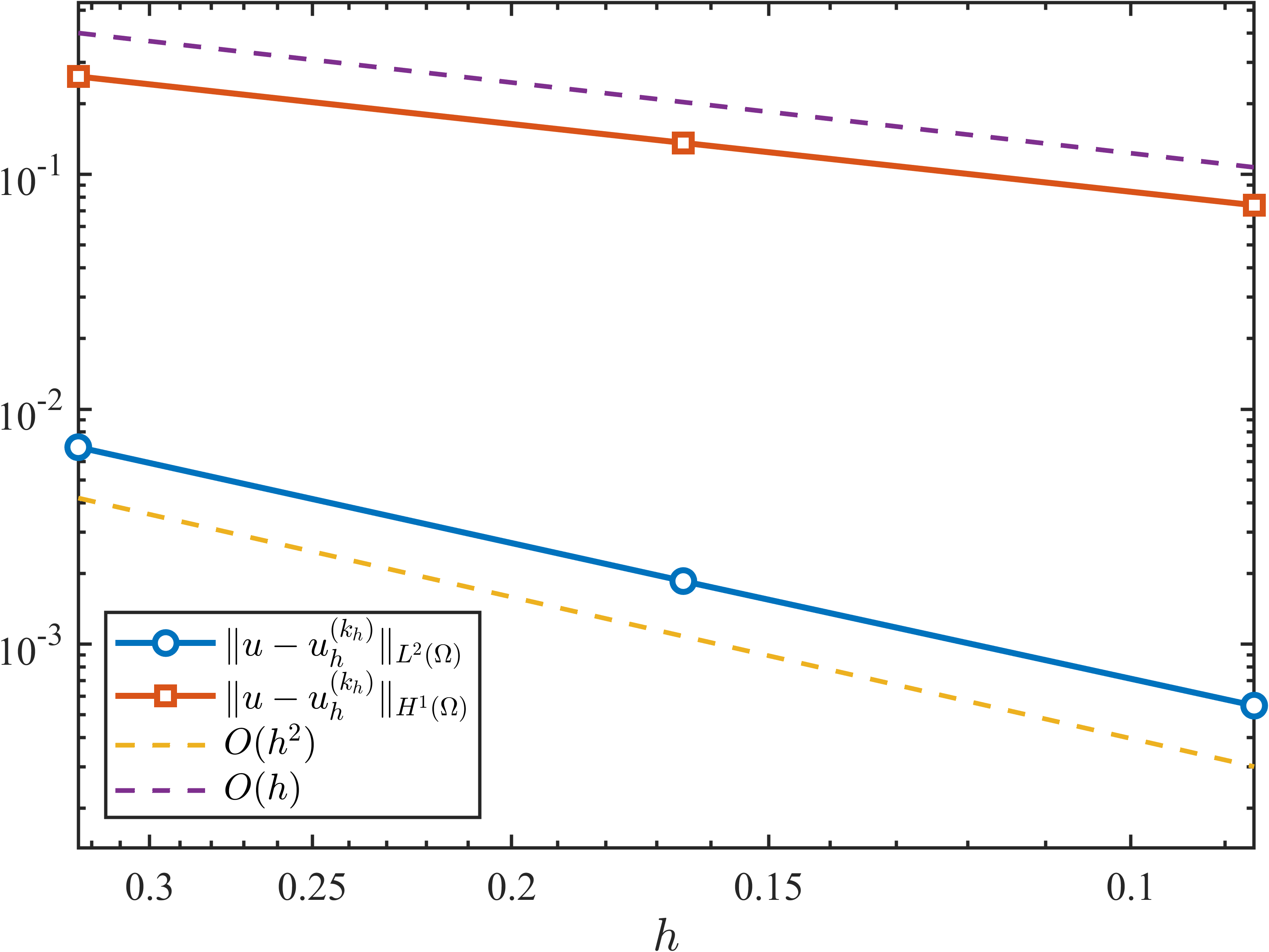}%
}
\caption{The finite element errors
$\|u-u_h^{(k_h)}\|_{H^1(\Omega)}$ and
$\|u-u_h^{(k_h)}\|_{L^2(\Omega)}$, together with the reference functions
$h$ and $h^2$, plotted on a log--log scale.}
\label{fig:error_rates}
\end{figure}

Finally, to assess the influence of the temperature dependence of the
conductivity tensor, we introduce
\begin{equation}\label{eq:cond_variation}
\delta_{\K}(x)
\coloneqq
\frac{
\left\|
\K\bigl(x,u_h^{(k_h)}(x)\bigr)-\K(x,0)
\right\|_{\mathrm F}
}{
\left\|\K(x,0)\right\|_{\mathrm F}
},
\qquad x\in\Omega,
\end{equation}
where $\|\cdot\|_{\mathrm F}$ denotes the Frobenius norm. Thus,
$\delta_{\K}(x)$ measures the relative departure of the conductivity
tensor evaluated at the computed temperature from its frozen value at
$u=0$.

\begin{figure}[H]
\centering

\subfigure[Example~$1$: $|u-u_h^{(k_h)}|$.\label{fig:error_ex1}]{%
    \includegraphics[width=0.47\textwidth]{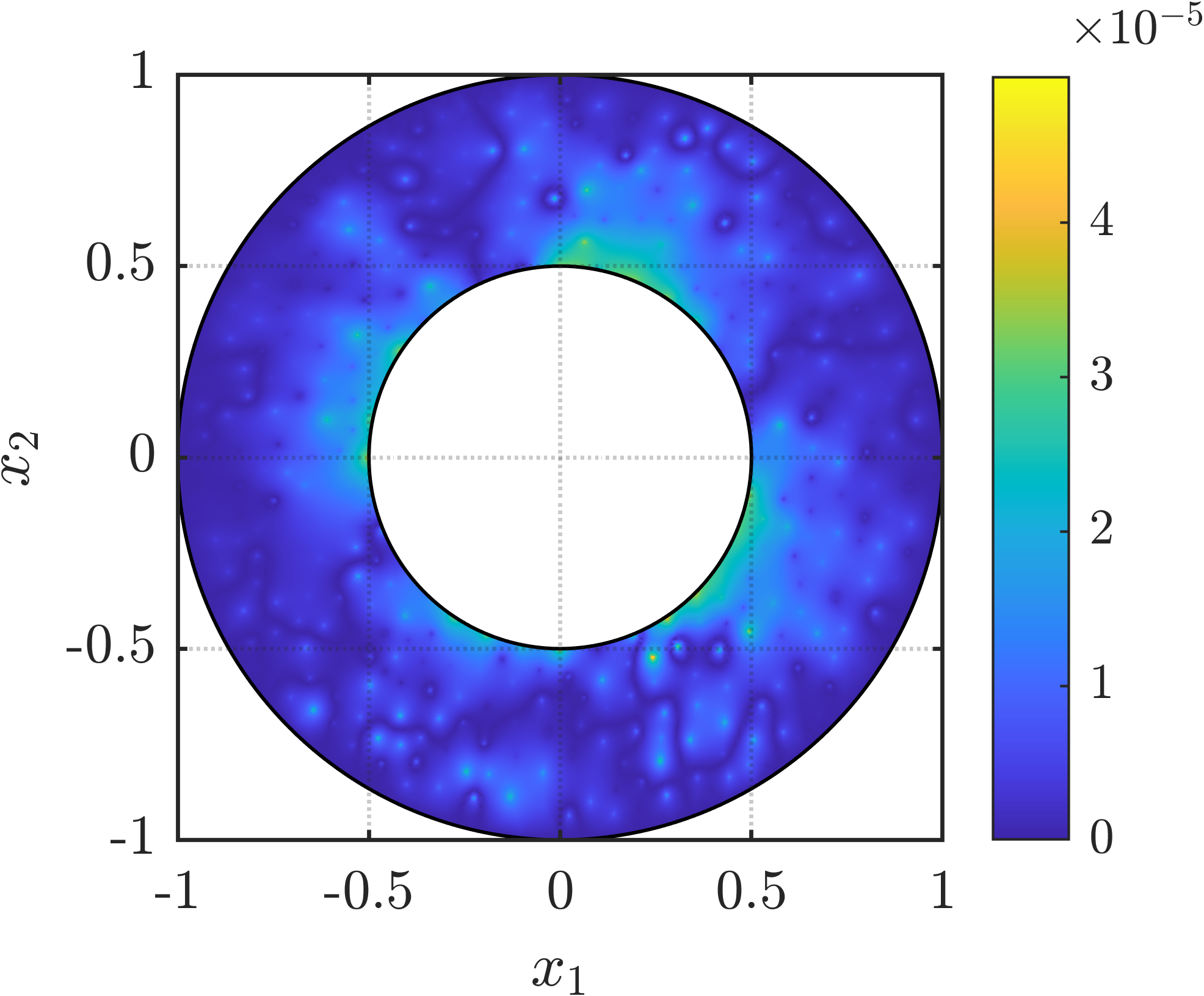}%
}
\hfill
\subfigure[Example~$1$: $\delta_{\K}$.\label{fig:conductivity_ex1}]{%
    \includegraphics[width=0.47\textwidth]{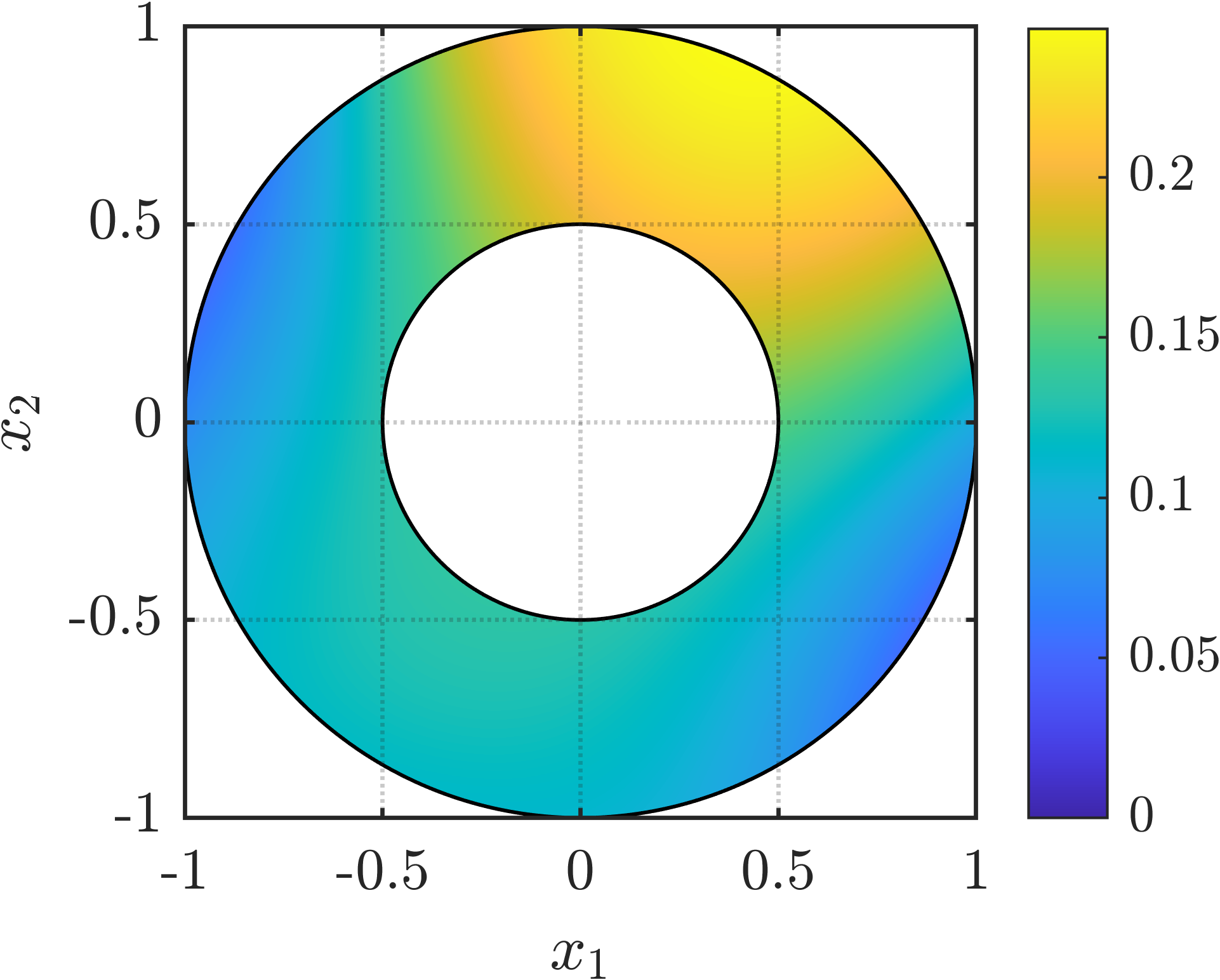}%
}

\vspace{0.5em}

\subfigure[Example~$2$: $|u-u_h^{(k_h)}|$.\label{fig:error_ex2}]{%
    \includegraphics[width=0.47\textwidth]{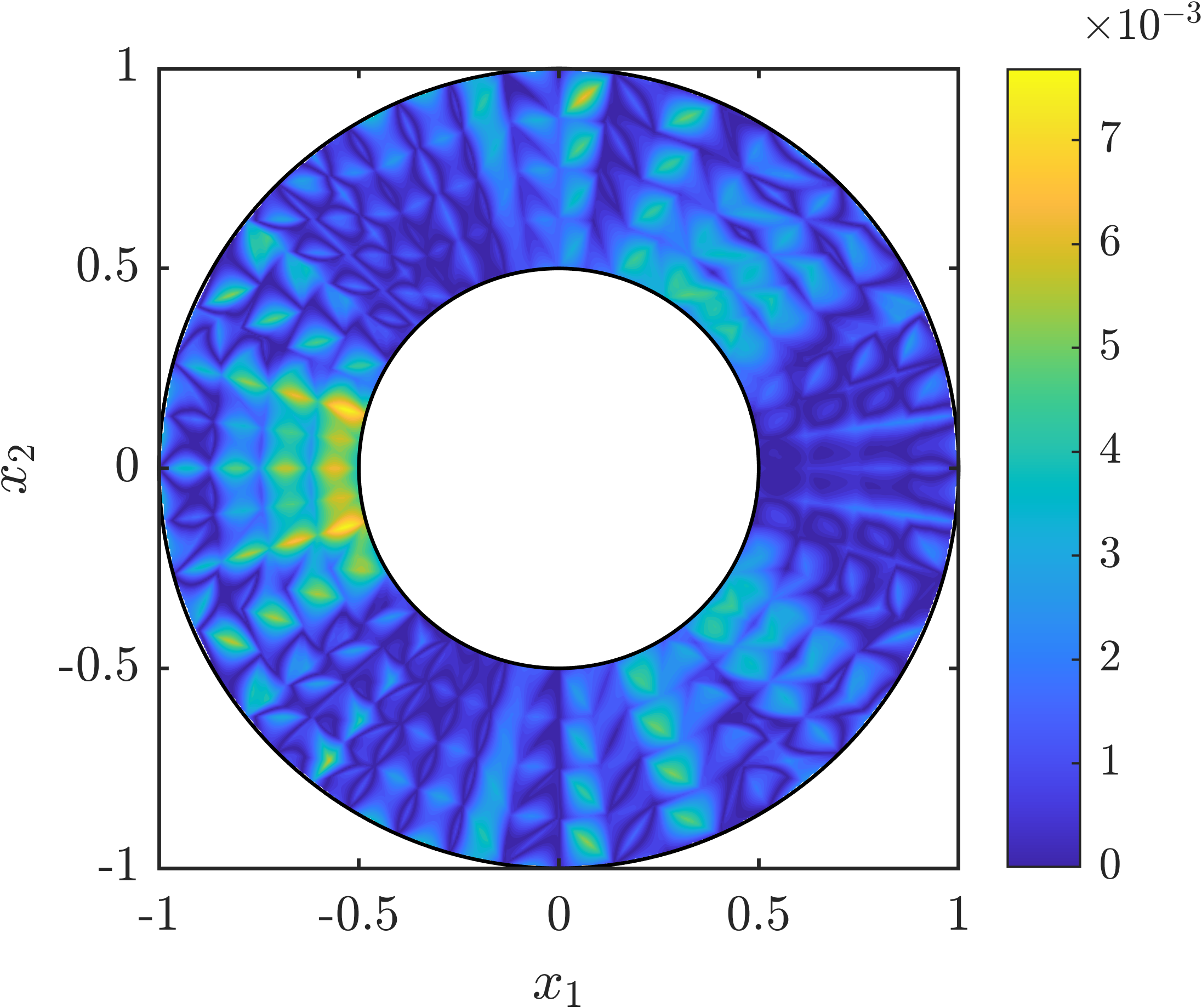}%
}
\hfill
\subfigure[Example~$2$: $\delta_{\K}$.\label{fig:conductivity_ex2}]{%
    \includegraphics[width=0.47\textwidth]{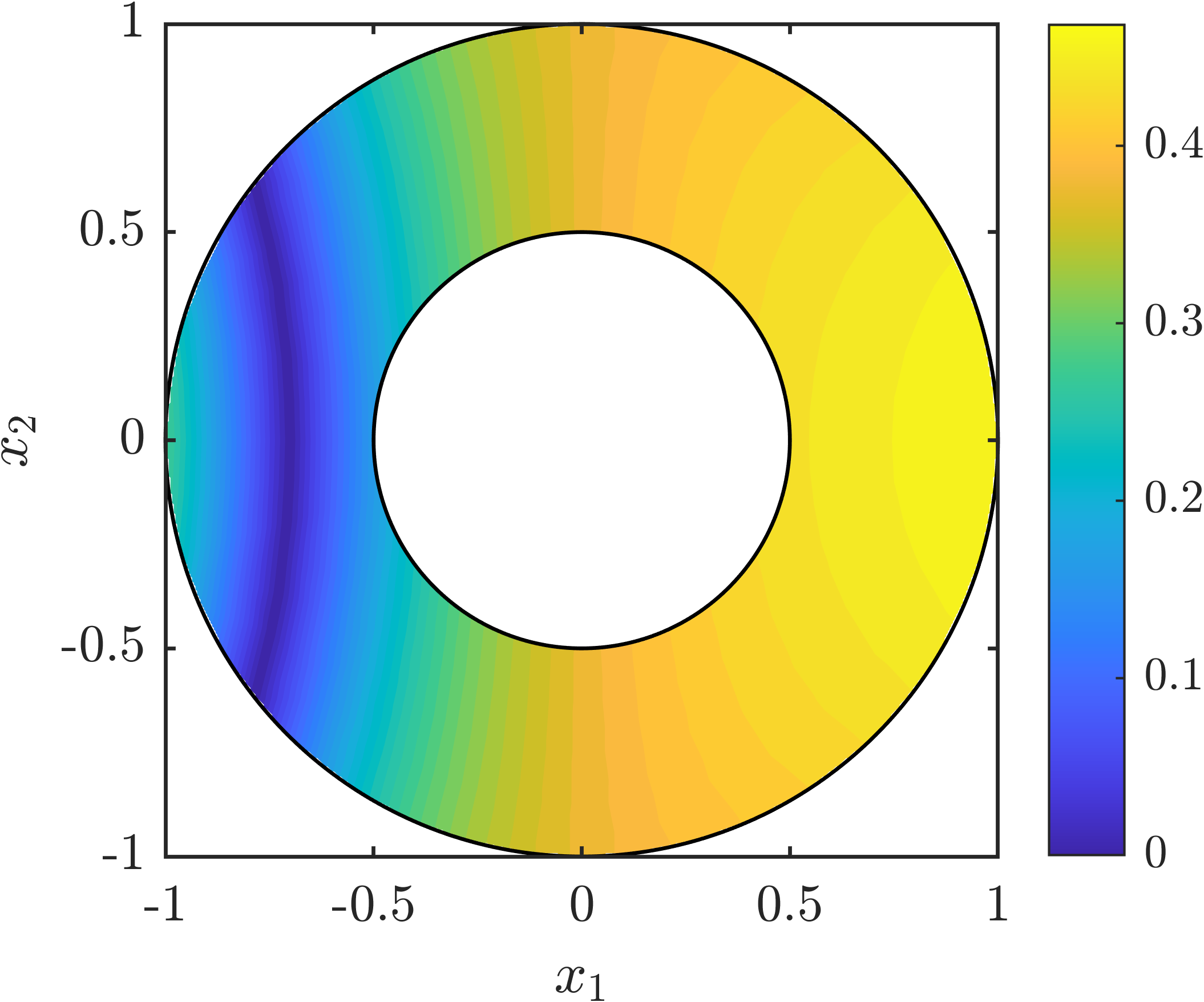}%
}

\caption{Spatial distributions obtained on the finest meshes. The top
and bottom rows correspond to Examples~1 and~2, respectively. The left
column shows the pointwise absolute error
$|u-u_h^{(k_h)}|$, while the right column shows the relative
temperature-induced variation $\delta_{\K}$ of the conductivity tensor,
defined in~\eqref{eq:cond_variation}. For Example~2, both quantities are
displayed on the central cross-section $x_3=0$ of the three-dimensional
domain.}
\label{fig:spatial_errors}
\end{figure}

Figure~\ref{fig:spatial_errors} displays the pointwise absolute error
$|u-u_h^{(k_h)}|$ together with the corresponding distribution of
$\delta_{\K}$, computed on the finest meshes reported in
Tables~\ref{tab:ex1} and~\ref{tab:ex2}, namely
$h=0.03$ for Example~1 and $h=0.08$ for Example~2. The colour scales show that the
computed temperature field changes the conductivity tensor by up to
approximately $25\%$ and $45\%$, respectively, confirming that both
examples exhibit a substantial nonlinear response, which is particularly
pronounced in Example~2. Moreover, the oscillatory spatial dependence of
the conductivity tensor is reflected in the localised oscillations and
peaks of the error fields, with the stronger and more spatially variable
nonlinearity in Example~2 being accompanied by a more pronounced error
pattern.

\section{Conclusion}
This work has studied conforming finite element approximations of a quasilinear elliptic model for anisotropic heat conduction with temperature-dependent conductivity. The conductivity tensor may depend on both position and temperature and may be matrix-valued, anisotropic and nonsymmetric, leading in general to a nonmonotone and nonpotential problem. In combination with inhomogeneous mixed Dirichlet--Neumann boundary conditions, this places the analysis beyond the standard homogeneous Dirichlet and monotonicity-based frameworks for nonlinear elliptic equations.

The error analysis was carried out without assuming global uniqueness of the nonlinear discrete problem. Instead, the estimates were obtained for arbitrary conforming piecewise linear Galerkin solutions, avoiding branch-selection arguments and additional discrete uniqueness conditions. By combining mixed-boundary elliptic regularity, nonlinear error estimates and an Aubin--Nitsche duality argument, we recovered the optimal first-order convergence in the $H^1$-norm and second-order convergence in the $L^2$-norm in this anisotropic, nonsymmetric and quasilinear setting.

A second main aspect of the paper was the connection between the finite element estimates and the actual nonlinear solver. Using a Newton--Kantorovich argument, we derived a computable local convergence criterion expressed through discrete quantities, namely the inverse Jacobian, an explicit Lipschitz bound for the discrete derivative and the Newton radius. This provides a practical certificate for the convergence of Newton's method to a discrete Galerkin solution. Moreover, a mesh-dependent stopping rule was shown to keep the algebraic error below the finite element error, so that the optimal convergence rates are preserved.

The numerical experiments in two and three dimensions support both parts of the theory. They confirm the predicted $H^1$- and $L^2$-convergence rates, illustrate the practical construction of a certified Newton initial iterate, and show that the Newton–Kantorovich condition can be verified in concrete computations.

\section*{Acknowledgements}

The author gratefully acknowledges Dr. Drago\c{s} Manea and Professor Liviu Marin for their valuable insights, suggestions, and stimulating discussions.


\appendix
\section{Appendix}

\begin{lem}\label{lem:grad-nodal-bound}

For every $v_h\in V_h$ with the corresponding nodal values $V = \{V_i\}_{i=1}^{N_h}$,
\begin{align}
    \|v_h\|_{L^2(\Omega)} &\le \left( \dfrac{(d+1) N_*}{d!} \right)^{1/2} h^{\frac{d}{2}} \| V \|_2\,, \label{eq:FEM-nodal-bound}\\[4pt]
    \|\nabla v_h\|_{L^2(\Omega)} &\le \left( \dfrac{(d+1)\sigma^2 N_*}{d!} \right)^{1/2} h^{\frac{d}{2} - 1} \| V \|_2\,, \label{eq:FEM-grad-nodal-bound}
\end{align}
where $\sigma$ is the shape-regularity parameter and $N_\ast$ is the maximum number of elements sharing a node.
\end{lem}

\begin{proof}

Let $T\in\mathcal T_h$ with the vertices  $N_{T,1},\dots,N_{T,d+1}$. On $T$ we have
\[
v_h|_T
=
\sum_{j=1}^{d+1} V_{T,j}\lambda_{T,j}\,, \qquad 
\nabla v_h|_T
=
\sum_{j=1}^{d+1} V_{T,j}\nabla\lambda_{T,j}\,,
\]
where $V_{T,j}$ are the nodal values and $\lambda_{T,j}$ the barycentric basis functions corresponding to $\{ N_{T, j}\}_{j=1}^{d + 1}$.
Using the triangle inequality and Cauchy-Schwarz,
\begin{align}
    |v_h|
    & \le
    \sum_{j=1}^{d+1}|V_{T,j}|\,|\lambda_{T,j}|
    \le
    \sqrt{d + 1} \left(\sum_{j=1}^{d+1} V_{T,j}^2 \, \lambda_{T,j} ^2\right)^{1 / 2}\label{eq:appendix-CS1} \\[4pt]
    |\nabla v_h|
    & \le
    \sum_{j=1}^{d+1}|V_{T,j}|\,|\nabla\lambda_{T,j}|
    \le
    \sqrt{d + 1} \left(\sum_{j=1}^{d+1} V_{T,j}^2 \, |\nabla \lambda_{T,j}|^2\right)^{1 / 2}\label{eq:appendix-CS2} 
\end{align}
Integrating \eqref{eq:appendix-CS1} over $T$, we have
\[
    \| v_h \|^2_{L^2(T)} \leq (d + 1) \sum_{j=1}^{d+1} V_{T,j}^2 \, \|\lambda_{T,j}\|^2_{L^2(T)}
\]
Since $0 \leq \lambda_{T, j} \leq 1$ on $T$, we have $\|\lambda_{T,j}\|^2_{L^2(T)} \leq |T|$. Employing the following bound for the volume of a $d$--simplex
\[
|T|
=
\frac{1}{d!}
\left|\det(N_{T,2} - N_{T, 1},\dots,N_{T, d + 1} - N_{T, 1})\right|
\le
\dfrac{h_T^d}{d!}
\le
\dfrac{h^d}{d!}\,,
\]
we arrive at
\[
    \| v_h \|^2_{L^2(T)} \leq \dfrac{(d + 1)}{d!} h^d \sum_{j=1}^{d+1} V_{T,j}^2 
\]
Summing over all elements, 
\[
\|v_h\|_{L^2(\Omega)}^2
\le
\dfrac{(d+1)}{d!} h^d
\sum_{T\in\mathcal T_h}\sum_{j=1}^{d+1}V_{T,j}^2 .
\]
Since at most $N_*$ elements meet at any node, we arrive at
\[
\|v_h\|_{L^2(\Omega)}^2 \le \dfrac{(d+1)N_*}{d!}  h^{d} \| V \|^2_2\,.
\]

To estimate now $\|\nabla v_h\|_{L^2(\Omega)}^2$, we return to \eqref{eq:appendix-CS2}, and note that the gradients of the barycentric basis functions $\lambda_{T,j}$ are constants on $T$, equal to the inverse of the altitude from vertex $N_{T, j}$ to the opposing face. Since every altitude is bounded below by the inradius, from the shape-regularity of the triangulations \eqref{eq:shape-regularity} we obtain that
\[
|\nabla\lambda_{T,j}|
\le
\frac{1}{\rho_T}
\le
\dfrac{\sigma}{h_T}
\qquad j=1,\dots,d+1.
\]
Hence,
\[
|\nabla v_h|
\le
\dfrac{\sigma\sqrt{d+1}}{h_T} \left(\sum_{j=1}^{d+1}V_{T,j}^2\right)^{1/2}.
\]
Since $\nabla v_h$ is constant on $T$, and employing once again the bound of the simplex volume $|T| \le h_T^d / d!$, 
\[
\|\nabla v_h\|_{L^2(T)}^2
=
|T|\,|\nabla v_h|^2 
\le
\dfrac{(d+1)\sigma^2}{d!}
h_T^{d-2}\sum_{j=1}^{d+1}V_{T,j}^2 .
\]
Summing over all elements and recalling that $h_T \le h$ for any $T \in \mathcal{T}_h$,
\[
\|\nabla v_h\|_{L^2(\Omega)}^2
\le
\dfrac{(d+1)\sigma^2}{d!} h^{d-2}
\sum_{T\in\mathcal T_h} \sum_{j=1}^{d+1}V_{T,j}^2
\le \dfrac{(d+1)\sigma^2 N_*}{d!}  h^{d - 2} \| V \|^2_2\,.
\]
\end{proof}

\clearpage
\printbibliography
\end{document}